\documentclass[a4paper]{amsart}
\usepackage[latin1]{inputenc}
\usepackage[T1]{fontenc}
\usepackage[french, english]{babel}
\usepackage{graphicx}
\usepackage{endnotes}
\usepackage{amsthm, amsmath}
\usepackage{amssymb}
\usepackage{xcolor}
\usepackage{graphics,graphicx}
\usepackage{tikz}
\usetikzlibrary{trees}
\usepackage[dvips,breaklinks]{hyperref}
\usepackage{textcomp}
\hypersetup{
pdfpagemode=UseOutlines, 
pdfstartview=Fit, 
pdffitwindow=true, 
pdfpagelayout=TwoColumnsRight,
pdftoolbar=true, 
pdfmenubar=true, 
bookmarksopen=false, 
bookmarksnumbered=true, 
colorlinks=true, 
linkcolor=blue, 
urlcolor=blue, 
anchorcolor=black, 
citecolor=green, 
frenchlinks=true, 
pdfborder={0 0 0} 
}

\numberwithin{equation}{section} 
\numberwithin{figure}{section} 
 \ifx\thesection\undefined
\newtheorem{thm}{Theorem}
\else
\newtheorem{thm}{Theorem}[section]
\fi
 \theoremstyle{definition}
 \newtheorem{defi}[thm]{Definition}
 \newtheorem{exple}[thm]{Example}

 \theoremstyle{plain}
 \newtheorem{prop}[thm]{Proposition}
 
 \newtheorem{cor}[thm]{Corollary}
 \newtheorem*{thmnonnum}{Theorem}
 
\newcommand{\hs}{\mathcal{H_{\mathcal{S}}}}
\newcommand{\hsprime}{\mathcal{H}'_{\mathcal{S}}}
\newcommand{\hsp}{\mathcal{H}^r_{\mathcal{S}}}
\newcommand{\hsa}{\mathcal{H}^e_{\mathcal{S}}}
\newcommand{\hspa}{\mathcal{H}^{re}_{\mathcal{S}}}
\newcommand{\hsc}{\mathcal{H}^h_{\mathcal{S}}}

\newcommand{\hsb}{\mathcal{H}_{\mathcal{S}_e,\mathcal{S}_v}}
\newcommand{\hsbp}{\mathcal{H}^{Ar}_{\mathcal{S}_e,\mathcal{S}_v}}
\newcommand{\hsbtp}{\mathcal{H}^{Br}_{\mathcal{S}_e,\mathcal{S}_v}}
\newcommand{\hsba}{\mathcal{H}^e_{\mathcal{S}_e,\mathcal{S}_v}}
\newcommand{\hsbpa}{\mathcal{H}^{re}_{\mathcal{S}_e,\mathcal{S}_v}}
\newcommand{\hsbc}{\mathcal{H}^h_{\mathcal{S}_e,\mathcal{S}_v}}

\newcommand{\h}{\mathcal{H}}
\newcommand{\hp}{\mathcal{H}^r}
\newcommand{\hpa}{\mathcal{H}^{re}}
\newcommand{\ha}{\mathcal{H}^e}
\newcommand{\hc}{\mathcal{H}^h}
\newcommand{\comm}{\operatorname{Comm}}
\newcommand{\perm}{\operatorname{Perm}}
\newcommand{\lie}{\operatorname{Lie}}
\newcommand{\prelie}{\operatorname{PreLie}}
\newcommand{\pasc}{\operatorname{Pasc}}
\newcommand{\assoc}{\operatorname{Assoc}}
\newcommand{\cycle}{\operatorname{Cycle}}

\newcommand{\hsy}{\mathcal{H}_{\comm -X, \cycle}}
\newcommand{\hsyp}{\mathcal{H}^{Ar}_{\comm -X,\cycle}}
\newcommand{\hsytp}{\mathcal{H}^{Br}_{\comm -X,\cycle}}
\newcommand{\hsya}{\mathcal{H}^e_{\comm -X,\cycle}}
\newcommand{\hsypa}{\mathcal{H}^{re}_{\comm -X,\cycle}}
\newcommand{\hsyc}{\mathcal{H}^h_{\comm -X,\cycle}}

\newcommand{\hsz}{\mathcal{H}_{\comm -X, \Sigma \lie}}
\newcommand{\hszp}{\mathcal{H}^{Ar}_{\comm -X, \Sigma \lie}}
\newcommand{\hsztp}{\mathcal{H}^{Br}_{\comm -X, \Sigma \lie}}
\newcommand{\hsza}{\mathcal{H}^e_{\comm -X, \Sigma \lie}}
\newcommand{\hszpa}{\mathcal{H}^{re}_{\comm -X, \Sigma \lie}}

\newcommand{\Ss}{S_{\mathcal{S}}}
\newcommand{\Sp}{S^r_{\mathcal{S}}}
\newcommand{\Sa}{S^e_{\mathcal{S}}}
\newcommand{\Spa}{S^{re}_{\mathcal{S}}}
\newcommand{\Sc}{S^h_{\mathcal{S}}}

\newcommand{\Ssw}{S_{\mathcal{S},W}}
\newcommand{\Spw}{S^r_{\mathcal{S},W}}
\newcommand{\Saw}{S^e_{\mathcal{S},W}}
\newcommand{\Spaw}{S^{re}_{\mathcal{S},W}}
\newcommand{\Scw}{S^h_{\mathcal{S},W}}

\title{Decorated hypertrees}

\author{Bérénice Oger}

\thanks{ Institut Camille Jordan, UMR 5208, Université Claude Bernard Lyon 1 \\
 Bât. Jean Braconnier n°101, 43 Bd du 11 novembre 1918, 69622 Villeurbanne Cedex \\
 \textbf{e-mail address:} oger@math.univ-lyon1.fr}

\begin{document}

\begin{abstract}
C. Jensen, J. McCammond and J. Meier have used weighted hypertrees to compute the Euler characteristic of a subgroup of the automorphism group of a free product. Weighted hypertrees also appear in the study of the homology of the hypertree poset. We link them to decorated hypertrees after a general study on decorated hypertrees, which we enumerate using box trees.
\end{abstract}

\maketitle


\tableofcontents

\section*{Introduction}

Hypergraphs are generalizations of graphs introduced by C. Berge in his book \cite{Berge} during the 1980's. Like graphs, they are defined by their vertices and edges, but the edges can contain more than two vertices. Hypertrees are hypergraphs in which there is one and only one walk between every pair of vertices. For a finite set $I$, we can endow the set of hypertrees on the vertex set $I$ with a structure of poset: given two hypertrees $H$ and $K$, $H \preceq K$ if each edge in $K$ is a subset of some edge in $H$. Recently, the hypertree poset has been used by C. Jensen, J. McCammond and J. Meier in their articles \cite{McCM}, \cite{JMcCM} and \cite{JMcCM2} for the study of a subgroup of the automorphism group of the free groups. The characteristic polynomial of the poset has been computed by F. Chapoton in the article \cite{ChHyp} and we have studied, in the article \cite{mar1}, the character of the action of the symmetric group on the homology and the Whitney homology of the hypertree poset. In another direction, K. Ebrahimi-Fard and D. Manchon explained in their article \cite{EFM} how hypertrees are organized in a combinatorial Hopf algebra which generalize the Connes-Kreimer Hopf algebra of trees. 

In this article, we consider hypertrees endowed with additional structures on edges, around vertices or with both. A typical example would be hypertrees with a cyclic order on the elements of every edge, and with a total order on the set of edges containing a given vertex, for every vertex. When the additional structure is on edges, the hypertrees are called \emph{decorated hypertrees}. When both kinds of additional structures are present, the hypertrees are called \emph{bi-decorated hypertrees}. Decorated hypertrees appear in the article \cite{JMcCM} of C. Jensen, J. McCammond and J. Meier for the computation of the Euler characteristic of a subgroup of the automorphism group of a free product. In his article \cite{ChHyp}, F. Chapoton has introduced a kind of bi-decorated hypertrees to compute the characteristic polynomial of the hypertree poset. Moreover, in the article \cite{mar1}, the character of the action of the symmetric group on the Whitney homology of the hypertree poset is related to another kind of bi-decorated hypertrees. Motivated by these examples, this present article aims at studying decorated hypertrees in general and showing that several known objects, such as the ones appearing in articles \cite{JMcCM}, \cite{mar1}, \cite{ChHyp} are indeed decorated hypertrees. 

For the general study of decorated hypertrees, we use the notion of combinatorial species and a trick often used for the study of trees: we root the hypertrees. Indeed, rooted trees and hypertrees are often easier to study than unrooted ones. Therefore, we define \emph{rooted hypertrees}, \emph{edge-pointed hypertrees} and \emph{rooted edge-pointed hypertrees}, which are hypertrees with respectively a distinguished vertex, a distinguished edge and a distinguished vertex inside a distinguished edge. We also define hollow hypertrees, which are hypertrees with one vertex labelled by $\#$ in one and only one edge. As decorated hypertrees and rooted and pointed variants of them are linked by the dissymetry principle \ref{principe de dissymétrie}, we study rooted and pointed decorated hypertrees to obtain results for decorated hypertrees. We get other relations between decorated hypertrees and rooted and pointed variants of them in Proposition \ref{décomp}. These relations are used to study the action of the symmetric group for some particular cases in §3. They are generalized for bi-decorated hypertrees in Proposition \ref{décompbi}.

 One of our main enumerative results is Theorem \ref{cor S}, which describes generating series of decorated hypertrees and all their pointed and rooted variants: 

\begin{thmnonnum} 
Given a species $\mathcal{S}$, the generating series of the species of edge-decorated rooted hypertrees, edge-decorated hollow hypertrees and edge-decorated hypertrees have the following expressions: 
\begin{equation*} 
\Sp\left(x\right)= \frac{x}{t} + \frac{1}{t} \sum_{n \geq 2} \sum_{k=1}^{n-1} E_\mathcal{S}\left(k,n-1\right)\left( nt \right)^{k} \frac{x^n}{n!},
\end{equation*}
\begin{equation*}
\Sc\left(x\right)= \sum_{n \geq 1} \sum_{k=1}^{n} E_\mathcal{S}\left(k,n\right)\left(nt\right)^{k-1} \frac{x^n}{n!},
\end{equation*}
\begin{equation*}
\Ss\left(x\right)= \frac{x}{t} + \sum_{n \geq 2} \sum_{k=1}^{n-1} E_\mathcal{S}\left(k,n-1\right)\left(nt\right)^{k-1} \frac{x^n}{n!},
\end{equation*}
where $E_\mathcal{S}\left(k,n\right)$ is the number of sets of $k$ sets on $n$ vertices with a $\mathcal{S}'$-structure on each of them and $\mathcal{S}'$ is the differential of the species $\mathcal{S}$.

The generating series of the species of rooted edge-decorated edge-pointed hypertrees and the species of edge-decorated edge-pointed hypertrees are related to the species of rooted edge-decorated hypertrees and the species of hollow edge-decorated hypertrees by the following relations: 
\begin{equation*}
\Spa\left(x\right)= t \Sp\left(x\right)\times \Sc\left(x\right),
\end{equation*}
\begin{equation*}
\Sa\left(x\right)= \Ss\left(x\right)+\Spa\left(x\right)-\Sp\left(x\right).
\end{equation*}
These formulas give the cardinality of all types of decorated hypertrees in terms of decorated sets.
\end{thmnonnum}
To prove this theorem, we introduce the notion of \emph{box trees}. A box tree is a graph whose vertices are sets of elements and whose edges are oriented from an element of a vertex to an other vertex. Moreover, the graph obtained by collapsing all the elements of a vertex to a single point has to be a rooted oriented tree.

\medskip

Our next aim is to prove that decorated hypertrees and bi-decorated hypertrees, for some specific choices of decorations, can be linked to known objects of various origins. For decorated hypertrees, two cases are studied in §3. The first one is the weighted hypertrees introduced in \cite{JMcCM}, which corresponds to hypertrees with edges decorated by the linear species $\widehat{\operatorname{PreLie}}$. The rooted variants of these hypertrees are hypertrees where every edge $e$ contains a single vertex and a rooted tree whose vertices are the other vertices of $e$. These hypertrees are in bijection with 2-coloured rooted trees, which are rooted trees whose edges are either blue or red. This bijection gives the character for the action of the symmetric group on the hollow variant of these decorated hypertrees. The relations between species found in §1 then allow to find the character for the other $\widehat{\operatorname{PreLie}}$-decorated hypertrees. The second type of decorated hypertrees studied in §3 are hypertrees with edges decorated by the linear species $\widehat{\operatorname{Lie}}$. We link them with the operad $\Lambda$ which was introduced in the article of F. Chapoton \cite{OpDiff}. 

For bi-decorated hypertrees, we link two types of them with the hypertree poset. The first case is hypertrees with their edges decorated by the set species $\comm$ and the neighbourhood of their vertices decorated by the cycle species. These hypertrees are called cyclic hypertrees and presented as hypertrees with a cyclic ordering of edges around every vertex in the article \cite{ChHyp}. The second case is hypertrees with their edges decorated by the set species $\comm$ and the neighbourhood of their vertices decorated by the species $\Sigma \lie $. We show that the character of the action of the symmetric group on the Whitney homology of the hypertree poset is the same as the character of the action of the symmetric group on these hypertrees.

\medskip

Let us now precisely describe the contents of the different sections. In §1, we recall the definition of hypergraphs and hypertrees and we define the species of edge-decorated hypertrees and its rooted and pointed variants before giving relations between them. In §2, we enumerate edge-decorated hypertrees using box trees. After this general study, we examine the case of $\widehat{\prelie}$-decorated hypertrees and $\widehat{\lie}$-decorated hypertrees by computing their cycle index series in §3. In the last section §4, we generalize edge-decorated hypertrees by also decorating neighbourhoods of vertices and link bi-decorated hypertrees with cyclic hypertrees and the hypertree poset.

We use the notation $\lvert E \rvert $ for the cardinality of a finite set $E$. We will use the language of species, defined in the book of F. Bergeron, G. Labelle and P. Leroux \cite{BLL} and recalled in \ref{def espèce}. The exponents $r$, $e$ and $re$ mean respectively rooted, edge-pointed and edge-pointed rooted.

We particularly thank V. Dotsenko for his useful proof of the Koszulness of the operad $\Lambda$.

\section{Description and relations of the edge-decorated hypertrees}

In this section, we introduce a type of weighted hypertrees, named \emph{edge-decorated hypertrees}, and give functional equations satisfied by them.

	\subsection{From hypergraphs to rooted and pointed hypertrees}

We first recall the definition of hypergraphs and hypertrees.
\begin{defi}
 A \emph{hypergraph (on a set $V$)} is an ordered pair $\left(V,E\right)$ where $V$ is a finite set and $E$ is a collection of parts of $V$ of cardinality at least two. The elements of $V$ are called \emph{vertices} and those of $E$ are called \emph{edges}.
\end{defi}

An example of hypergraph is presented in figure \ref{exple hypergraph}.
		
\begin{figure} 
\centering
\begin{tikzpicture}[scale=1]
\draw (0,0) -- (0,1) node[midway, left]{A};
\draw (0,1) -- (1,1) node[midway, above]{B};
\draw[black, fill=gray!40] (1,1) -- (1,0) -- (2,0) -- (2,1) -- (1,1);
\draw (1.5,0.5) node{C};
\draw[black, fill=gray!40] (2,0) -- (2,1) -- (3,1) -- (2,0);
\draw (2.30,0.70) node{D};
\draw[black, fill=white] (0,0) circle (0.2);
\draw[black, fill=white] (0,1) circle (0.2);
\draw[black, fill=white] (1,1) circle (0.2);
\draw[black, fill=white] (1,0) circle (0.2);
\draw[black, fill=white] (2,0) circle (0.2);
\draw[black, fill=white] (2,1) circle (0.2);
\draw[black, fill=white] (3,1) circle (0.2);
\draw (0,0) node{$4$};
\draw (0,1) node{$7$};
\draw (1,1) node{$6$};
\draw (1,0) node{$5$};
\draw (2,1) node{$1$};
\draw (2,0) node{$2$};
\draw (3,1) node{$3$};
\end{tikzpicture}
\caption{A hypergraph on $\{1,2,3,4,5,6,7\}$. \label{exple hypergraph}}
\end{figure}
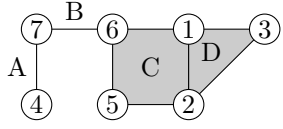

\begin{defi} Let $H=\left(V,E\right)$ be a hypergraph, $v$ and $w$ two vertices of $H$.
		 A \emph{walk from $v$ to $w$ in $H$} is an alternating sequence of vertices and edges 
		 \begin{equation*}\left(v = v_1, e_1, v_2, \ldots, e_n, v_{n+1} = w\right)
		 \end{equation*} 
		 where \begin{itemize}
		 \item for all $i$ in $\{1,\ldots,n+1\}$, $v_i \in V$, $e_i \in E$ 
		 \item and for all $i$ in $\{1,\ldots,n\}$, $\{v_i,v_{i+1}\} \subseteq e_i$. 
		 \end{itemize}
\end{defi}

\begin{exple} In figure \ref{exple hypergraph}, there are several walks from $4$ to $2$ as, for instance: 

$\left(4,A,7,B,6,C,2\right)$ and $\left(4,A,7,B,6,C,1,D,2\right)$.
\end{exple}

In this article, we are interested in a special type of hypergraphs: hypertrees.

\begin{defi} A \emph{hypertree} is a non-empty hypergraph $H$ such that, given any vertices $v$ and $w$ in $H$, 
\begin{itemize}
\item there exists a walk from $v$ to $w$ in $H$ with distinct edges $e_i$, i.e. $H$ is \emph{connected},
\item and this walk is unique, i.e. $H$ has \emph{no cycle}. 
 \end{itemize}

The pair $H=\left(V,E\right)$ is called \emph{hypertree on $V$}. If $V$ is the set $ \{ 1, \ldots, n \}$, then $H$ is called a \emph{hypertree on $n$ vertices}.
\end{defi}

Let us remark that we can consider hypertrees as bipartite trees, where one type of vertex is labelled by elements of $V$, the other type is not labelled and every vertex of this type stands for an edge of the hypertree. An example of hypertree is presented in figure \ref{exh}.

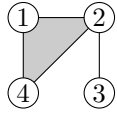
\begin{figure} 
\centering 
\begin{tikzpicture}[scale=1]
\draw[black, fill=gray!40] (0,0) -- (1,1) -- (0,1) -- (0,0);
\draw (1,0) -- (1,1);
\draw[black, fill=white] (0,0) circle (0.2);
\draw[black, fill=white] (0,1) circle (0.2);
\draw[black, fill=white] (1,1) circle (0.2);
\draw[black, fill=white] (1,0) circle (0.2);
\draw (0,0) node{$4$};
\draw (0,1) node{$1$};
\draw (1,1) node{$2$};
\draw (1,0) node{$3$};
\end{tikzpicture}
\caption{A hypertree on $\{1,2,3,4\}$. \label{exh}}
\end{figure}

We now recall rooted and pointed variants of hypertrees.

\begin{defi} A \emph{rooted hypertree} is a hypertree $H$ together with a vertex $s$ of $H$. The hypertree $H$ is said to be \emph{rooted at $s$} and $s$ is called the root of $H$.
\end{defi}

An example of rooted hypertrees is presented in figure \ref{exr}.
\begin{figure} 
\centering 
\begin{tikzpicture}[scale=1]
\draw[black, fill=gray!40] (1,1) -- (1,0) -- (2,0) -- (2,1) -- (1,1);
\draw[black, fill=gray!40] (2,0) -- (3,0) -- (3,1) -- (2,0);
\draw[black, fill=gray!40] (0,0) -- (1,0) -- (0,1) -- (0,0);
\draw (2,2) -- (2,1);
\draw[black, fill=white] (0,0) circle (0.2);
\draw[black, fill=white] (0,1) circle (0.2);
\draw[black, fill=white] (1,1) circle (0.2);
\draw[black, fill=white] (1,0) circle (0.3);
\draw[black, fill=white] (1,0) circle (0.2);
\draw[black, fill=white] (2,0) circle (0.2);
\draw[black, fill=white] (2,1) circle (0.2);
\draw[black, fill=white] (3,1) circle (0.2);
\draw[black, fill=white] (3,0) circle (0.2);
\draw[black, fill=white] (2,2) circle (0.2);
\draw (0,0) node{$9$};
\draw (0,1) node{$8$};
\draw (1,1) node{$2$};
\draw (1,0) node{$1$};
\draw (2,1) node{$3$};
\draw (2,0) node{$4$};
\draw (3,1) node{$6$};
\draw (2,2) node{$5$};
\draw (3,0) node{$7$};
\end{tikzpicture}
\caption{A hypertree on nine vertices, rooted at $1$. \label{exr}}
\end{figure}
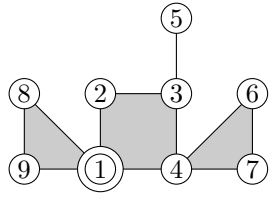

\begin{defi} An \emph{edge-pointed hypertree} is a hypertree $H$ together with an edge $e$ of $H$. The hypertree $H$ is said to be \emph{pointed at $e$}.
\end{defi}

An example of edge-pointed hypertree is presented in figure \ref{exple edge-pointed}.

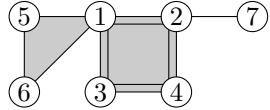
\begin{figure} 
\centering 
\begin{tikzpicture}[scale=1]
\draw[black, fill=gray!40] (1,1) -- (1,0) -- (2,0) -- (2,1) -- (1,1);
\draw[black, fill=gray!40] (1.1,0.9) -- (1.1,0.1) -- (1.9,0.1) -- (1.9,0.9) -- (1.1,0.9);
\draw[black, fill=gray!40] (0,0) -- (1,1) -- (0,1) -- (0,0);
\draw (3,1) -- (2,1);
\draw[black, fill=white] (0,0) circle (0.2);
\draw[black, fill=white] (0,1) circle (0.2);
\draw[black, fill=white] (1,1) circle (0.2);
\draw[black, fill=white] (1,0) circle (0.2);
\draw[black, fill=white] (2,0) circle (0.2);
\draw[black, fill=white] (2,1) circle (0.2);
\draw[black, fill=white] (3,1) circle (0.2);
\draw (0,0) node{$6$};
\draw (0,1) node{$5$};
\draw (1,1) node{$1$};
\draw (1,0) node{$3$};
\draw (2,1) node{$2$};
\draw (2,0) node{$4$};
\draw (3,1) node{$7$};
\end{tikzpicture}
\caption{\label{exple edge-pointed} A hypertree on seven vertices, pointed at the edge $\{1,2,3,4\}$.}
\end{figure}

\begin{defi} An \emph{edge-pointed rooted hypertree} is a hypertree $H$ on at least two vertices, together with an edge $a$ of $H$ and a vertex $v$ of $a$. The hypertree $H$ is said to be \emph{pointed at $a$ and rooted at $s$}.
\end{defi}

An example of edge-pointed rooted hypertree is presented in figure \ref{exrp}.

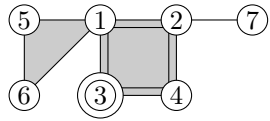
\begin{figure} 
\centering 
\begin{tikzpicture}[scale=1]
\draw[black, fill=gray!40] (1,1) -- (1,0) -- (2,0) -- (2,1) -- (1,1);
\draw[black, fill=gray!40] (1.1,0.9) -- (1.1,0.1) -- (1.9,0.1) -- (1.9,0.9) -- (1.1,0.9);
\draw[black, fill=gray!40] (0,0) -- (1,1) -- (0,1) -- (0,0);
\draw (3,1) -- (2,1);
\draw[black, fill=white] (1,0) circle (0.3);
\draw[black, fill=white] (0,0) circle (0.2);
\draw[black, fill=white] (0,1) circle (0.2);
\draw[black, fill=white] (1,1) circle (0.2);
\draw[black, fill=white] (1,0) circle (0.2);
\draw[black, fill=white] (2,0) circle (0.2);
\draw[black, fill=white] (2,1) circle (0.2);
\draw[black, fill=white] (3,1) circle (0.2);
\draw (0,0) node{$6$};
\draw (0,1) node{$5$};
\draw (1,1) node{$1$};
\draw (1,0) node{$3$};
\draw (2,1) node{$2$};
\draw (2,0) node{$4$};
\draw (3,1) node{$7$};
\end{tikzpicture}
\caption{A hypertree on seven vertices, pointed at edge $\{1,2,3,4\}$ and rooted at $3$. \label{exrp}}
\end{figure}

\begin{defi}
A \emph{hollow hypertree} with vertex set $I$ is a hypertree on the set $\{\#\} \cup I$, such that the vertex labelled by $\#$, called the gap, belongs to one and only one edge. 
\end{defi}

An example of hollow hypertree is presented in figure \ref{exhh}.

\begin{figure} 
\centering 
\begin{tikzpicture}[scale=1]
\draw[black, fill=gray!40] (2,0) -- (3,-1) -- (4,0) -- (3,1) -- (2,0);
\draw[black, fill=gray!40] (1,0) -- (1,1) -- (2,1) -- (2,0) -- (1.5,-1) -- (1,0);
\draw (0,0) -- (1,0);
\draw[black, fill=white] (0,0) circle (0.2);
\draw[black, fill=white] (1,1) circle (0.2);
\draw[black, fill=white] (1,0) circle (0.2);
\draw[black, fill=white] (2,0) circle (0.2);
\draw[black, fill=white] (2,1) circle (0.2);
\draw[black, fill=white] (3,1) circle (0.2);
\draw[black, fill=white] (3,-1) circle (0.2);
\draw[black, fill=white] (4,0) circle (0.2);
\draw[black, fill=white] (1.5,-1) circle (0.2);
\draw (0,0) node{$5$};
\draw (1,1) node{$2$};
\draw (1,0) node{$1$};
\draw (2,1) node{$3$};
\draw (2,0) node{$4$};
\draw (3,1) node{$6$};
\draw (1.5,-1) node{$\#$};
\draw (3,-1) node{$8$};
\draw (4,0) node{$7$};
\end{tikzpicture}
\caption{A hollow hypertree on eight vertices. The hollow edge is the edge $\{1,2,3,4\}$. \label{exhh}}
\end{figure}
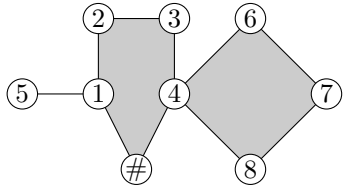

\subsection{Definitions of decorated hypertrees}

From a hypertree, we can define what we call an \emph{edge-decorated hypertree}. This definition uses the following notion of \emph{species}: 

\begin{defi} \label{def espèce} A \emph{species} $\operatorname{F}$ is a functor from the category of finite sets and bijections to the category of finite sets. To a finite set $I$, the species $\operatorname{F}$ associates a finite set $\operatorname{F}\left(I\right)$ independent from the nature of $I$. 
\end{defi}

\begin{exple} \label{exple espèces}
\begin{itemize}
\item The map which associates to a finite set $I$ the set of total orders on $I$ is a species, called the linear order species and denoted by $\assoc$.
\item The map which associates to a finite set $I$ the set $\{I\}$ is a species, called the set species and denoted by $\mathcal{E}$.
\item The map which associates to a finite set $I$ the set of cyclic orders on $I$ is a species, called the cycle species and denoted by $\cycle$.
\item The map which associates to a finite set $I$ the set of hypertrees on $I$ is a species, called the hypertrees species and denoted by $\h$.
\item The map which associates to a finite set $I$ the set $I$ is a species, called the the pointed set species and denoted by ${\perm}$.
\item The map defined for every finite set $I$ by: 
\begin{equation*}
I \mapsto \left\lbrace
\begin{array}{ll}
\{I\} & \text{if } \lvert I \rvert \geq 1, \\
\emptyset & \text{otherwise},
\end{array} \right.
\end{equation*}
is a species denoted by $\comm$.
\item The map defined for every finite set $I$ by: 
\begin{equation*}
I \mapsto \left\lbrace
\begin{array}{ll}
\{I\} & \text{if } \lvert I \rvert = 1, \\
\emptyset & \text{otherwise},
\end{array} \right.
\end{equation*}
is a species, called the singleton species and denoted by $X$.
\end{itemize}
\end{exple}

The map which associates to a finite set $I$ the set of rooted (resp. edge-pointed, edge-pointed rooted, hollow) hypertrees on $I$ is a species, called the rooted (resp. edge-pointed, edge-pointed rooted, hollow) hypertrees species and denoted by $\hp$ (resp. $\ha$, $\hpa$, $\hc$).

\begin{defi}[\cite{BLL}] 
Let $F$ and $G$ be two species. An \emph{isomorphism} of $F$ to $G$ is a family of bijections $\alpha_U: F\left(U\right) \mapsto G\left(U\right)$, where $U$ is a finite set, such that for any finite set $V$, any bijection $\sigma: U \mapsto V$ and any $F$-structure $s \in F\left(u\right)$, the following equality holds: 
\begin{equation*}
\sigma \centerdot \alpha_U\left(s\right)=\alpha_V\left(\sigma \centerdot s\right)
\end{equation*}
where the $\centerdot$ stands for the action of the symmetric group on the structure.

The two species $F$ and $G$ are then said to be isomorphic.
\end{defi}

Having defined species, we can now give the definition of an \emph{edge-decorated hypertree}.

\begin{defi}
Given a species $\mathcal{S}$, an \emph{edge-decorated (edge-pointed) hypertree} is obtained from an (edge-pointed) hypertree $H$ by choosing for every edge $e$ of $H$ an element of $\mathcal{S}\left(V_e\right)$, where $V_e$ is the set of vertices in the edge $e$.
\end{defi}

Given a species $\mathcal{S}$, the map which associates to a finite set $I$ the set of edge-decorated hypertrees on $I$ is a species, called the $\mathcal{S}$-edge-decorated hypertrees species and denoted by $\hs$. The map which associates to a finite set $I$ the set of edge-decorated edge-pointed hypertrees on $I$ is a species, called the $\mathcal{S}$-edge-decorated edge-pointed hypertrees species and denoted by $\hsa$. When the species used is obvious, we will omit to write it.

\begin{exple} We consider two different decorations of the same hypertree $H$.

An edge-decorated hypertree obtained from $H$ by decorating its edges with the cycle species $\cycle$ is drawn in the left part of figure \ref{2dec}. 

An edge-decorated hypertree obtained from $H$ by decorating its edges with the pointed set species $Perm$ is drawn in the right part of figure \ref{2dec}. The vertex of each edge with a star $*$ next to it is the pointed vertex of the edge: for instance, the pointed vertex of $\{9,11\}$ is $9$, the pointed vertex of $\{9,10\}$ is $10$, the pointed vertex of $\{8,9\}$ is $9$ and the pointed vertex of $\{12,13,1,8\}$ is $12$. 

\begin{figure} 
\centering 
\begin{tikzpicture}[scale=1]
\draw[black, fill=gray!40] (1,2) -- (1,3) -- (2,3) -- (2,2) -- (1,2);
\draw[black, fill=gray!40] (2,0) -- (3,0) -- (2,1) -- (2,0);
\draw[black, fill=gray!40] (2,1) -- (2,2) -- (3,2) -- (2,1);
\draw (0,0) -- (1,1);
\draw (1,0) -- (1,1);
\draw (1,2) -- (1,1);
\draw (2,3) -- (1,3);
\draw (3,2) -- (4,2);
\draw (3,2) -- (3,1);
\draw[>=stealth, ->] (0.30,0.20) -- (0.80,0.70);
\draw[>=stealth, ->] (0.70,0.80) -- (0.20,0.30);
\draw[>=stealth, ->] (1.07,0.25) -- (1.07,0.75);
\draw[>=stealth, ->] (0.93,0.75) -- (0.93,0.25);
\draw[>=stealth, ->] (1.07,1.25) -- (1.07,1.75);
\draw[>=stealth, ->] (0.93,1.75) -- (0.93,1.25);
\draw[>=stealth, ->] (3.07,1.25) -- (3.07,1.75);
\draw[>=stealth, ->] (2.93,1.75) -- (2.93,1.25);
\draw[>=stealth, ->] (3.25,2.07) -- (3.75,2.07);
\draw[>=stealth, ->] (3.75,1.93) -- (3.25,1.93);
\draw[>=stealth, ->] (2.75,1.90) -- (2.25,1.90);
\draw[>=stealth, ->] (2.10,1.75) -- (2.10,1.25);
\draw[>=stealth, ->] (2.20,1.30) -- (2.70,1.80);
\draw[>=stealth, ->] (2.10,0.25) -- (2.10,0.75);
\draw[>=stealth, ->] (2.20,0.70) -- (2.70,0.20);
\draw[>=stealth, ->] (2.75,0.07) -- (2.25,0.07);
\draw[>=stealth, ->] (1.20,2.20) -- (1.80,2.80);
\draw[>=stealth, ->] (1.75,2.90) -- (1.25,2.90);
\draw[>=stealth, ->] (1.20,2.80) -- (1.80,2.20);
\draw[>=stealth, ->] (1.75,2.15) -- (1.25,2.15);

\draw[black, fill=white] (0,0) circle (0.2);
\draw[black, fill=white] (1,0) circle (0.2);
\draw[black, fill=white] (1,1) circle (0.2);
\draw[black, fill=white] (1,2) circle (0.2);
\draw[black, fill=white] (1,3) circle (0.2);
\draw[black, fill=white] (2,0) circle (0.2);
\draw[black, fill=white] (2,1) circle (0.2);
\draw[black, fill=white] (2,2) circle (0.2);
\draw[black, fill=white] (2,3) circle (0.2);
\draw[black, fill=white] (3,0) circle (0.2);
\draw[black, fill=white] (3,1) circle (0.2);
\draw[black, fill=white] (3,2) circle (0.2);
\draw[black, fill=white] (4,2) circle (0.2);
\draw (0,0) node{$11$};
\draw (1,1) node{$9$};
\draw (1,0) node{$10$};
\draw (1,2) node{$8$};
\draw (1,3) node{$12$};
\draw (2,0) node{$4$};
\draw (2,1) node{$3$};
\draw (2,2) node{$1$};
\draw (2,3) node{$13$};
\draw (3,0) node{$5$};
\draw (3,1) node{$6$};
\draw (3,2) node{$2$};
\draw (4,2) node{$7$};

\draw [black, fill=gray!40] (6,2) -- (6,3) -- (7,3) -- (7,2) -- (6,2);
\draw[black, fill=gray!40] (7,0) -- (8,0) -- (7,1) -- (7,0);
\draw[black, fill=gray!40] (7,1) -- (7,2) -- (8,2) -- (7,1);
\draw (6,0) -- (6,1);
\draw (6,2) -- (6,1);
\draw (8,2) -- (9,2);
\draw (8,2) -- (8,1);
\draw (5,0) -- (6,1);
\draw (6,2) -- (6,1);
\draw (6,0) -- (6,1);
\draw[black, fill=white] (5,0) circle (0.2);
\draw[black, fill=white] (6,0) circle (0.2);
\draw[black, fill=white] (6,1) circle (0.2);
\draw[black, fill=white] (6,2) circle (0.2);
\draw[black, fill=white] (6,3) circle (0.2);
\draw[black, fill=white] (7,0) circle (0.2);
\draw[black, fill=white] (7,1) circle (0.2);
\draw[black, fill=white] (7,2) circle (0.2);
\draw[black, fill=white] (7,3) circle (0.2);
\draw[black, fill=white] (8,0) circle (0.2);
\draw[black, fill=white] (8,1) circle (0.2);
\draw[black, fill=white] (8,2) circle (0.2);
\draw[black, fill=white] (9,2) circle (0.2);
\draw (5,0) node{$11$};
\draw (6,1) node{$9$};
\draw (5.7,0.9) node{$*$};
\draw (6.1,0.3) node{$*$};
\draw (6.1,1.3) node{$*$};
\draw (6,0) node{$10$};
\draw (6,2) node{$8$};
\draw (6,3) node{$12$};
\draw (6.25,2.75) node{$*$};
\draw (7,0) node{$4$};
\draw (7.2,0.2) node{$*$};
\draw (7,1) node{$3$};
\draw (7.1,1.3) node{$*$};
\draw (7,2) node{$1$};
\draw (7,3) node{$13$};
\draw (8,0) node{$5$};
\draw (8,1) node{$6$};
\draw (8.1,1.3) node{$*$};
\draw (8,2) node{$2$};
\draw (9,2) node{$7$};
\draw (8.7,2.1) node{$*$};
\end{tikzpicture}
\caption{Edge-decorated hypertrees. \label{2dec}}
\end{figure}
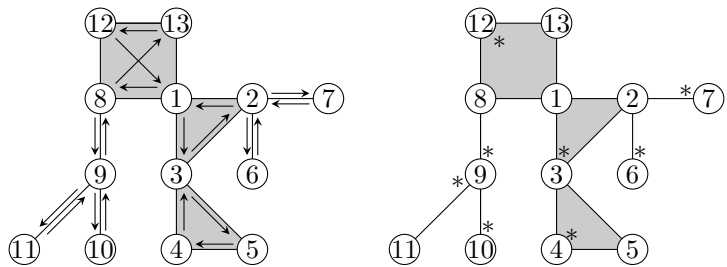

In figure \ref{eprt}, we now decorate the edge-pointed hypertree of figure \ref{exple edge-pointed} with the species of trees, which associates to a finite set $I$ the set of trees with vertex set $I$. The edges are represented by rectangles on the scheme, with the pointed edge represented by the double rectangle.

\begin{figure} 
\centering 
\begin{tikzpicture}[scale=1]
\draw[black, fill=gray!40, rounded corners] (-0.5,-0.5) rectangle (1.5,2.5);
\draw[black, fill=gray!40, rounded corners] (-0.4,-0.4) rectangle (1.4,2.4);
\draw (1,0) -- (1,1);
\draw (1,1) -- (0,2);
\draw (1,1) -- (1,2);
\draw[black, fill=gray!40, rounded corners] (0.5,1.7) rectangle (3.5,2.3);
\draw (1,2) -- (2,2);
\draw (3,2) -- (2,2);
\draw[black, fill=gray!40, rounded corners] (0.5,-0.3) rectangle (2.5,0.3);
\draw (1,0) -- (2,0);
\draw[densely dotted] (1,0) -- (1,1);
\draw[densely dotted] (1,1) -- (1,2);
\draw[black, fill=white] (1,0) circle (0.2);
\draw[black, fill=white] (2,0) circle (0.2);
\draw[black, fill=white] (1,1) circle (0.2);
\draw[black, fill=white] (3,2) circle (0.2);
\draw[black, fill=white] (0,2) circle (0.2);
\draw[black, fill=white] (1,2) circle (0.2);
\draw[black, fill=white] (2,2) circle (0.2);
\draw (3,2) node{$6$};
\draw (2,2) node{$5$};
\draw (1,2) node{$1$};
\draw (1,1) node{$3$};
\draw (1,0) node{$2$};
\draw (0,2) node{$4$};
\draw (2,0) node{$7$};
\end{tikzpicture}
\caption{Edge-pointed decorated hypertree. \label{eprt}}
\end{figure}
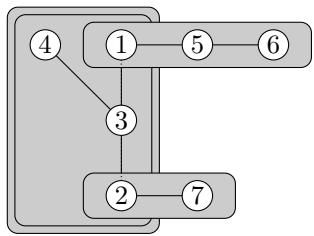

\end{exple}

We now give definitions for rooted or hollow variants of edge-decorated hypertrees.
			
To define \emph{edge-decorated rooted hypertrees}, we will need the following operation on species: 

\begin{defi} Let $F$ be a species. The differential of $F$ is the species defined as follow: 
\begin{equation*}
F'\left(I\right)=F\left(I \sqcup \{\bullet \} \right).
\end{equation*}
\end{defi}

\begin{exple} Considering the examples of species given in Example \ref{exple espèces}, their differentials are the following:
\begin{itemize}
\item The differential of the cycle species $\cycle$ is the linear order species $\assoc$. 
\item The differential of the set species $\mathcal{E}$ is $\mathcal{E}$.
\item The differential of the pointed set species $\perm$ is the species $\perm'$ such that, for all finite set $I$, the set $\perm'\left(I\right)$ is equal to the set $\mathcal{E}\left(I\right) \cup \perm\left(I\right)$. 
\end{itemize}
\end{exple}

We can now define the decoration for edges of rooted variants of hypertrees. There is a definition of edge-decorated rooted, edge-decorated edge-pointed rooted or edge-decorated hollow hypertrees analogous to the one of edge-decorated hypertrees and edge-decorated edge-pointed hypertrees: 

\begin{defi}
Given a species $\mathcal{S}$, an \emph{edge-decorated (edge-pointed) rooted hypertree} (resp. \emph{edge-decorated hollow hypertree}) is obtained from an (edge-pointed) rooted (resp. hollow) hypertree $H$ by choosing for every edge $e$ of $H$ an element of $\mathcal{S}\left(V_e\right)$, where $V_e$ is the set of vertices in the edge $e$.
\end{defi}

In rooted or hollow hypertrees, there is one distinguished vertex in every edge. Therefore, using the definition of the differential of a species, we obtain the following equivalent definition: 

\begin{defi} \label{1.23}
Let us consider a rooted (resp. edge-pointed rooted, resp. hollow) hypertree $H$. Given an edge $e$ of $H$, there is one vertex of $e$ which is the nearest from the root (resp. the gap) of $H$ in $e$: let us call it the \emph{petiole} of $e$. Then, an \emph{edge-decorated rooted hypertree} (resp. \emph{edge-decorated rooted edge-pointed hypertree}, resp. \emph{edge-decorated hollow hypertree}) is obtained from the hypertree $H$ by choosing for every edge $e$ of $H$ an element in the set $\mathcal{S}'\left(V_e^l\right)$, where the set $V_e^l$ is the set of vertices of $e$ different from the petiole.
\end{defi}

The map which associates to a finite set $I$ the set of edge-decorated rooted hypertrees on $I$ is a species, called the $\mathcal{S}$-edge-decorated rooted hypertrees species and denoted by $\hsp$.

The map which associates to a finite set $I$ the set of edge-decorated edge-pointed rooted hypertrees on $I$ is a species, called the $\mathcal{S}$-edge-decorated edge-pointed rooted hypertrees species and denoted by $\hspa$.

The map which associates to a finite set $I$ the set of edge-decorated hollow hypertrees on $I$ is a species, called the $\mathcal{S}$-edge-decorated hollow hypertrees species and denoted by $\hsc$.

\begin{exple} \label{exemple rooted} We give an example of an edge-decorated rooted hypertree. The left-side part of figure \ref{rdh} is the one obtained when edges are decorated by the cycles species. The right-side part of figure \ref{rdh} is the same rooted hypertree in which the set of elements of an edge different from the petiole forms a list. The two pictures are two representations of the same hypertree.

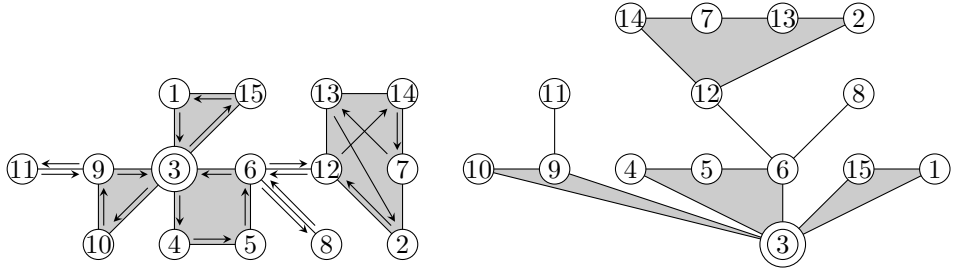
\begin{figure} 
\centering 
\begin{tikzpicture}[scale=1]
\draw[black, fill=gray!40] (1,0) -- (1,1) -- (2,1) -- (1,0);
\draw[black, fill=gray!40] (2,0) -- (3,0) -- (3,1) -- (2,1) -- (2,0);
\draw[black, fill=gray!40] (2,1) -- (2,2) -- (3,2) -- (2,1);
\draw[black, fill=gray!40] (4,1) -- (4,2) -- (5,2) -- (5,1) -- (5,0) -- (4,1);
\draw (0,1) -- (1,1);
\draw (3,1) -- (4,1);
\draw (3,1) -- (4,0);

\draw[>=stealth, ->] (0.25,0.93) -- (0.75,0.93);
\draw[>=stealth, ->] (0.75,1.07) -- (0.25,1.07);
\draw[>=stealth, ->] (3.25,1.07) -- (3.75,1.07);
\draw[>=stealth, ->] (3.75,0.93) -- (3.25,0.93);
\draw[>=stealth, ->] (3.15,0.75) -- (3.75,0.15);
\draw[>=stealth, ->] (3.85,0.25) -- (3.26,0.85);
\draw[>=stealth, ->] (4.85,0.25) -- (4.26,0.85);
\draw[>=stealth, ->] (4.2,1.2) -- (4.8,1.8);
\draw[>=stealth, ->] (4.93,1.75) -- (4.93,1.25);
\draw[>=stealth, ->] (4.8,1.2) -- (4.2,1.8);
\draw[>=stealth, ->] (4.1,1.7) -- (4.9,0.3);
\draw[>=stealth, ->] (2.75,0.93) -- (2.35,0.93);
\draw[>=stealth, ->] (1.25,0.93) -- (1.65,0.93); 
\draw[>=stealth, ->] (2.07,0.65) -- (2.07,0.25); 
\draw[>=stealth, ->] (2.25,0.07) -- (2.75,0.07); 
\draw[>=stealth, ->] (2.93,0.25) -- (2.93,0.75); 
\draw[>=stealth, ->] (1.65,0.75) -- (1.2,0.3); 
\draw[>=stealth, ->] (1.07, 0.25) -- (1.07,0.75); 
\draw[>=stealth, ->] (2.07,1.75) -- (2.07,1.35); 
\draw[>=stealth, ->] (2.2,1.3) -- (2.75,1.85); 
\draw[>=stealth, ->] (2.75,1.93) -- (2.25,1.93); 

\draw[black, fill=white] (2,1) circle (0.3);
\draw[black, fill=white] (0,1) circle (0.2);
\draw[black, fill=white] (1,0) circle (0.2);
\draw[black, fill=white] (1,1) circle (0.2);
\draw[black, fill=white] (2,0) circle (0.2);
\draw[black, fill=white] (2,1) circle (0.2);
\draw[black, fill=white] (2,2) circle (0.2);
\draw[black, fill=white] (3,0) circle (0.2);
\draw[black, fill=white] (3,1) circle (0.2);
\draw[black, fill=white] (3,2) circle (0.2);
\draw[black, fill=white] (4,0) circle (0.2);
\draw[black, fill=white] (4,1) circle (0.2);
\draw[black, fill=white] (4,2) circle (0.2);
\draw[black, fill=white] (5,0) circle (0.2);
\draw[black, fill=white] (5,1) circle (0.2);
\draw[black, fill=white] (5,2) circle (0.2);
\draw (0,1) node{$11$};
\draw (1,1) node{$9$};
\draw (1,0) node{$10$};
\draw (2,0) node{$4$};
\draw (2,1) node{$3$};
\draw (2,2) node{$1$};
\draw (3,0) node{$5$};
\draw (3,1) node{$6$};
\draw (3,2) node{$15$};
\draw (4,0) node{$8$};
\draw (4,1) node{$12$};
\draw (4,2) node{$13$};
\draw (5,0) node{$2$};
\draw (5,1) node{$7$};
\draw (5,2) node{$14$};

\draw[black, fill=gray!40] (10,0) -- (6,1) -- (7,1) -- (10,0);
\draw[black, fill=gray!40] (10,0) -- (8,1) -- (9,1) -- (10,1) -- (10,0);
\draw[black, fill=gray!40] (10,0) -- (11,1) -- (12,1) -- (10,0);
\draw[black, fill=gray!40] (9,2) -- (8,3) -- (9,3) -- (10,3) -- (11,3) -- (9,2);
\draw (7,1) -- (7,2);
\draw (10,1) -- (9,2);
\draw (10,1) -- (11,2);

\draw[black, fill=white] (10,0) circle (0.3);
\draw[black, fill=white] (10,0) circle (0.2);
\draw[black, fill=white] (6,1) circle (0.2);
\draw[black, fill=white] (7,1) circle (0.2);
\draw[black, fill=white] (7,2) circle (0.2);
\draw[black, fill=white] (8,1) circle (0.2);
\draw[black, fill=white] (9,1) circle (0.2);
\draw[black, fill=white] (10,1) circle (0.2);
\draw[black, fill=white] (11,1) circle (0.2);
\draw[black, fill=white] (12,1) circle (0.2);
\draw[black, fill=white] (9,2) circle (0.2);
\draw[black, fill=white] (11,2) circle (0.2);

\draw[black, fill=white] (8,3) circle (0.2);
\draw[black, fill=white] (9,3) circle (0.2);
\draw[black, fill=white] (10,3) circle (0.2);
\draw[black, fill=white] (11,3) circle (0.2);

\draw (10,0) node{$3$};
\draw (6,1) node{$10$};
\draw (7,1) node{$9$};
\draw (8,1) node{$4$};
\draw (9,1) node{$5$};
\draw (10,1) node{$6$};
\draw (11,1) node{$15$};
\draw (12,1) node{$1$};
\draw (7,2) node{$11$};
\draw (11,2) node{$8$};
\draw (9,2) node{$12$};

\draw (8,3) node{$14$};
\draw (9,3) node{$7$};
\draw (10,3) node{$13$};
\draw (11,3) node{$2$};
\end{tikzpicture}
\caption{Rooted hypertrees decorated by the cycle species. \label{rdh}}
\end{figure}

We now decorate the edges of an edge-pointed rooted hypertree with the species of lists, represented on the left-side of figure \ref{edrh}. The small numbers inside the edges near every vertex is the number of the vertex in the list. The derivative of the species of lists is the species of the unions of two lists. Hence this decorated hypertree is equivalent to the right-side hypertree of the figure where, for all edges, the set of vertices different from the petiole is separated into two lists. We draw a dashed line for the separation in an edge between the two lists. In the edge $\{2,7\}$, the second list is empty.

\begin{figure} 
\centering 
\begin{tikzpicture}[scale=1]
\draw[black, fill=gray!40] (1,1) -- (1,0) -- (2,0) -- (2,1) -- (1,1);
\draw[black, fill=gray!40] (1.1,0.9) -- (1.1,0.1) -- (1.9,0.1) -- (1.9,0.9) -- (1.1,0.9);
\draw[black, fill=gray!40] (0,0) -- (1,1) -- (0,1) -- (0,0);
\draw (3,1) -- (2,1);
\draw[black, fill=white] (1,0) circle (0.3);
\draw[black, fill=white] (0,0) circle (0.2);
\draw[black, fill=white] (0,1) circle (0.2);
\draw[black, fill=white] (1,1) circle (0.2);
\draw[black, fill=white] (1,0) circle (0.2);
\draw[black, fill=white] (2,0) circle (0.2);
\draw[black, fill=white] (2,1) circle (0.2);
\draw[black, fill=white] (3,1) circle (0.2);
\draw (0,0) node{$6$};
\draw (0,1) node{$5$};
\draw (1,1) node{$1$};
\draw (1,0) node{$3$};
\draw (2,1) node{$2$};
\draw (2,0) node{$4$};
\draw (3,1) node{$7$};

\draw (0.25,0.75) node{$_1$};
\draw (0.65,0.80) node{$_2$};
\draw (0.15,0.30) node{$_3$};
\draw (1.30,0.75) node{$_1$};
\draw (1.75,0.30) node{$_3$};
\draw (1.75,0.75) node{$_4$};
\draw (1.30,0.30) node{$_2$};
\draw (2.25,0.75) node{$_2$};
\draw (2.75,0.75) node{$_1$};

\draw[black, fill=gray!40] (6,0) -- (5,1) -- (7,1) -- (8,1) -- (6,0);
\draw[black, fill=gray!40] (5,1) -- (4,2) -- (6,2) -- (5,1);
\draw (8,1) -- (7,2);
\draw[dashed] (6,0) -- (6,1);
\draw[dashed] (8,1) -- (8,2);
\draw[dashed] (5,1) -- (5,2);
\draw[black, fill=white] (6,0) circle (0.3);
\draw[black, fill=white] (6,0) circle (0.2);
\draw[black, fill=white] (7,1) circle (0.2);
\draw[black, fill=white] (8,1) circle (0.2);
\draw[black, fill=white] (5,1) circle (0.2);
\draw[black, fill=white] (4,2) circle (0.2);
\draw[black, fill=white] (6,2) circle (0.2);
\draw[black, fill=white] (7,2) circle (0.2);
\draw (4,2) node{$5$};
\draw (5,1) node{$1$};
\draw (6,0) node{$3$};
\draw (8,1) node{$2$};
\draw (7,1) node{$4$};
\draw (7,2) node{$7$};
\draw (6,2) node{$6$};
\end{tikzpicture}
\caption{Rooted hypertrees decorated by the species of lists. \label{edrh}}
\end{figure}
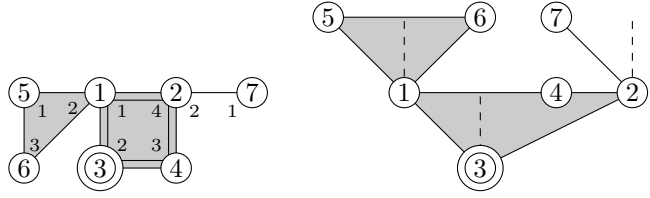

In figure \ref{edrh2}, we give an example of an edge-decorated hollow hypertree decorated by the species of non-empty pointed sets $\perm$. As previously, the left part of the example is obtained with the first definition of edge-decorated hollow hypertrees. We draw a star $*$ next to the pointed vertex. The right part is obtained with the second definition.

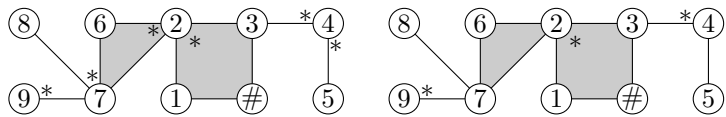
\begin{figure} 
\centering 
\begin{tikzpicture}[scale=1]
\draw[black, fill=gray!40] (1,0) -- (2,1) -- (1,1) -- (1,0);
\draw[black, fill=gray!40] (2,0) -- (2,1) -- (3,1) -- (3,0) -- (2,0);
\draw (0,0) -- (1,0);
\draw (0,1) -- (1,0);
\draw (3,1) -- (4,1);
\draw (4,1) -- (4,0);
\draw[black, fill=white] (0,0) circle (0.2);
\draw[black, fill=white] (0,1) circle (0.2);
\draw[black, fill=white] (1,1) circle (0.2);
\draw[black, fill=white] (1,0) circle (0.2);
\draw[black, fill=white] (2,0) circle (0.2);
\draw[black, fill=white] (2,1) circle (0.2);
\draw[black, fill=white] (3,1) circle (0.2);
\draw[black, fill=white] (3,0) circle (0.2);
\draw[black, fill=white] (4,0) circle (0.2);
\draw[black, fill=white] (4,1) circle (0.2);
\draw (0,0) node{$9$};
\draw (0.3,0.1) node{$*$};
\draw (0.9,0.3) node{$*$};
\draw (0,1) node{$8$};
\draw (1,1) node{$6$};
\draw (1,0) node{$7$};
\draw (1.7,0.9) node{$*$};
\draw (2,1) node{$2$};
\draw (2,0) node{$1$};
\draw (2.25,0.75) node{$*$};
\draw (3,1) node{$3$};
\draw (3,0) node{$\#$};
\draw (4,1) node{$4$};
\draw (3.7,1.1) node{$*$};
\draw (4,0) node{$5$};
\draw (4.1,0.7) node{$*$};

\draw[black, fill=gray!40] (6,0) -- (7,1) -- (6,1) -- (6,0);
\draw[black, fill=gray!40] (7,0) -- (7,1) -- (8,1) -- (8,0) -- (7,0);
\draw (5,0) -- (6,0);
\draw (5,1) -- (6,0);
\draw (8,1) -- (9,1);
\draw (9,1) -- (9,0);
\draw[black, fill=white] (5,0) circle (0.2);
\draw[black, fill=white] (5,1) circle (0.2);
\draw[black, fill=white] (6,1) circle (0.2);
\draw[black, fill=white] (6,0) circle (0.2);
\draw[black, fill=white] (7,0) circle (0.2);
\draw[black, fill=white] (7,1) circle (0.2);
\draw[black, fill=white] (8,1) circle (0.2);
\draw[black, fill=white] (8,0) circle (0.2);
\draw[black, fill=white] (9,0) circle (0.2);
\draw[black, fill=white] (9,1) circle (0.2);
\draw (5,0) node{$9$};
\draw (5.3,0.1) node{$*$};
\draw (5,1) node{$8$};
\draw (6,1) node{$6$};
\draw (6,0) node{$7$};
\draw (7,1) node{$2$};
\draw (7,0) node{$1$};
\draw (7.25,0.75) node{$*$};
\draw (8,1) node{$3$};
\draw (8,0) node{$\#$};
\draw (9,1) node{$4$};
\draw (8.7,1.1) node{$*$};
\draw (9,0) node{$5$};
\end{tikzpicture}
\caption{Rooted hypertrees decorated by the species of pointed sets. \label{edrh2}}
\end{figure}
 \end{exple}
 
We now generalize these definitions to linear species.

\subsection{Generalization to linear species and first relations with other trees}

Definition \ref{1.23} motivates the introduction of rooted or hollow hypertrees with edges decorated by species which are not the differential of a species but of a linear species.

\begin{defi} A \emph{linear species} $\operatorname{F}$ is a functor from the category of finite sets and bijections to the category of vector spaces. To a finite set $I$, the species $\operatorname{F}$ associates a vector space $\operatorname{F}\left(I\right)$ independent from the nature of $I$. 
\end{defi}

We define the differential of linear species: 

\begin{defi} Let $F$ be a linear species. The differential of $F$ is the species defined as follow: 
\begin{equation*}
F'\left(I\right)=F\left(I \sqcup \{\bullet \} \right).
\end{equation*}
\end{defi}

We can define the same operations as for species on linear species.

We can now generalize the decoration of edges of hypertrees: 

\begin{defi}
Given a linear species $\mathcal{S}$, an \emph{edge-decorated (edge-pointed) hypertree} is obtained from an (edge-pointed) hypertree $H$ by choosing for every edge $e$ of $H$ an element of $\mathcal{S}\left(V_e\right)$, where $V_e$ is the set of vertices in the edge $e$.

This decoration is multi-linear.
\end{defi}

We define similarly edge-decorated (edge-pointed) rooted hypertrees and edge-decorated hollow hypertrees.

Let us consider a species $\mathcal{S}$. When there exists a species or a linear species $\mathcal{F}$ such that $\mathcal{F}'=\mathcal{S}$, we denote it by $\widehat{\mathcal{S}}$.

For every cyclic or anticyclic operad $\mathcal{C}$, there always exists a linear species $\widehat{\mathcal{C}}$. This proves the existence of the linear species $\widehat{
\prelie}$, $\widehat{\lie}$, $\widehat{\perm}$ and $\widehat{\assoc}$. This case of operads is examined in the article of F. Chapoton \cite{ChAOp} and in the article of E. Getzler and M. Kapranov \cite{CochGK}.

We will now say \emph{species} for species or linear species.

The notion of edge-decorated hollow hypertrees can be related with different objects according to the decoration. For instance, if we consider the linear species $\widehat{\perm}$ whose derivative is the species $\perm$, the associated edge-decorated hollow hypertrees will be related with what is called fat trees.

In figure \ref{exemple hollow}, we first give an example of $\widehat{\operatorname{Perm}}$-edge-decorated hollow hypertree, where $\operatorname{Perm}$ is the species of non-empty pointed sets. The pointed vertex of each edge is marked by a star $*$ next to it. The petiole of the edge $\{2,6,7\}$ is $2$. Remark that vertices $8$, $9$, $4$ and $5$ are necessarily pointed because they are alone with the petiole in their edge. 

\begin{figure} 
\centering 
\begin{tikzpicture}[scale=1]
\draw[black, fill=gray!40] (1,0) -- (2,1) -- (1,1) -- (1,0);
\draw[black, fill=gray!40] (2,0) -- (2,1) -- (3,1) -- (3,0) -- (2,0);
\draw (0,0) -- (1,0);
\draw (0,1) -- (1,0);
\draw (3,1) -- (4,1);
\draw (4,1) -- (4,0);
\draw[black, fill=white] (0,0) circle (0.2);
\draw[black, fill=white] (0,1) circle (0.2);
\draw[black, fill=white] (1,1) circle (0.2);
\draw[black, fill=white] (1,0) circle (0.2);
\draw[black, fill=white] (2,0) circle (0.2);
\draw[black, fill=white] (2,1) circle (0.2);
\draw[black, fill=white] (3,1) circle (0.2);
\draw[black, fill=white] (3,0) circle (0.2);
\draw[black, fill=white] (4,0) circle (0.2);
\draw[black, fill=white] (4,1) circle (0.2);
\draw (0,0) node{$9$};
\draw (0.3,0.1) node{$*$};
\draw (0.3,1) node{$*$};
\draw (0,1) node{$8$};
\draw (1,1) node{$6$};
\draw (1,0) node{$7$};
\draw (1.1,0.3) node{$*$};
\draw (2,1) node{$2$};
\draw (2,0) node{$1$};
\draw (2.25,0.25) node{$*$};
\draw (3,1) node{$3$};
\draw (3,0) node{$\#$};
\draw (4,1) node{$4$};
\draw (3.7,1.1) node{$*$};
\draw (4,0) node{$5$};
\draw (4.1,0.3) node{$*$};
\end{tikzpicture}
\caption{A $\widehat{\operatorname{Perm}}$-edge-decorated hollow hypertree.\label{exemple hollow}}
\end{figure}

\begin{defi}[ \cite{Zas}] 
A \emph{fat tree on a set $V$} is a partition of $V$, whose parts are called \emph{vertices}, together with edges linking elements of different vertices, such that: 
\begin{itemize}
\item a \emph{walk} on the fat tree is an alternating sequence $\left(a_1, b_1, c_1, a_2, \ldots, c_n\right)$, where for every $i$, $a_i$ and $c_i$ are elements of different vertices and $b_i$ is an edge between $a_i$ and $c_i$, and for every $i$ between $1$ and $n-1$, $c_i$ and $a_{i+1}$ are elements of the same vertex;
\item For every pair of elements of different vertices $\left(a,c\right)$, there exists one and only one walk from $a$ to $c$.
\end{itemize}

A \emph{rooted fat tree} is a fat tree with a distinguished element called the \emph{root}.
\end{defi}

In figure \ref{exemple fat}, an example of a rooted fat tree is presented. The root is circled.

\begin{figure} 
\centering 
\begin{tikzpicture}[scale=1]
\draw[black, rounded corners] (-0.3,-0.3) rectangle(2.3,0.3);
\draw[black, rounded corners] (-0.3,0.7) rectangle(1.3,1.3);
\draw[black, rounded corners] (1.7,0.7) rectangle(2.3,1.3);
\draw[black, rounded corners] (-0.3,1.7) rectangle(0.3,2.3);
\draw[black, rounded corners] (0.7,1.7) rectangle(1.3,2.3);
\draw[black, rounded corners] (1.7,1.7) rectangle(2.3,2.3);
\draw (1,0.85) -- (1,0.15);
\draw (2,0.85) -- (2,0.15);
\draw (1,1.85) -- (1,1.15);
\draw (0,1.85) -- (1,1.15);
\draw (2,1.85) -- (2,1.15);
\draw[black, fill=white] (0,0) circle (0.2);
\draw (0,0) node{$1$};
\draw (1,0) node{$2$};
\draw (2,0) node{$3$};
\draw (2,1) node{$4$};
\draw (2,2) node{$5$};
\draw (0,1) node{$6$};
\draw (1,1) node{$7$};
\draw (1,2) node{$8$};
\draw (0,2) node{$9$};
\end{tikzpicture}
\caption{ \label{exemple fat} A rooted fat tree.}
\end{figure}
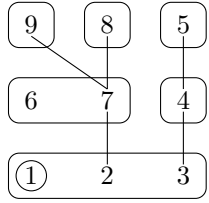

\begin{prop}
The species of $\widehat{\perm}$-edge-decorated hollow hypertrees is isomorphic to the species of rooted fat trees.
\end{prop}

\begin{proof}
First, remark that $\widehat{\operatorname{Perm}}$-edge-decorated hollow hypertrees are hollow hypertrees where in every edge, the set of vertices different from the petiole is a pointed set.

Let us consider a rooted fat tree $FT$ on a finite set $V$. We call \emph{petiole} of an edge $e$ the closest vertex of $e$ from the root and \emph{end} the other vertex. 

Let us consider a vertex $v$ of the rooted fat tree. We form an edge of a hypergraph by considering the set of elements of $V$ in the vertex $v$ and the petiole of the edge linking this vertex to the root, if the root is not in the considered edge. Let us call $E$ the set of such edges. The hypergraph $H=\left(E,V\right)$ thus obtained is a hypertree because every path in $H$ was in $FT$. 
If we put a gap in the edge of $E$ containing the root of $FT$, we obtain a hollow hypertree. Moreover, the end of every edge in $FT$ gives a pointed element in the set of vertices of an edge of $H$ different from the petiole. This structure is the same as the one obtained by decorating edges with $\widehat{\operatorname{Perm}}$.

Conversely, let us now take a $\widehat{\operatorname{Perm}}$-edge-decorated hollow hypertree $H$. We consider a structure whose vertices are the edges of $H$ without their petiole and without the gap, with the pointed elements of every edge of $H$ linking to their petiole by an edge in the new structure: this is a fat tree $FT$. By distinguishing the vertex of $FT$ obtained from the hollow edge of $H$, we obtain a rooted fat tree.
\end{proof}

\begin{exple}
The two previous figures \ref{exemple hollow} and \ref{exemple fat} are related by the bijection of the proof.
\end{exple}

	\subsection{Relations}

		\subsubsection{Dissymetry principle}

The reader may consult the book \cite[Chapter 2.3]{BLL} for a deeper explanation on the dissymmetry principle. In a general way, a \emph{dissymmetry principle} is the use of a natural center to obtain the expression of a non pointed species in terms of pointed species. An example of this principle is the use of the center of a tree to express unrooted trees in terms of rooted trees. 

We will consider the following weight on any hypertree (pointed or not, rooted or not, hollow or not): 

\begin{defi}
The weight of a hypertree $H$ with edge set $E$ is given by: 
\begin{equation*}
W_t\left(H\right)=t^{\lvert E \rvert -1}.
\end{equation*}
\end{defi}

The expression of the hypertrees species in term of pointed and rooted hypertrees species is the following:

\begin{prop}[\cite{mar1}] \label{prindi}
The species of hypertrees and the one of rooted hypertrees are related by: 
\begin{equation}
\h+\hpa=\hp+\ha.
\end{equation}
\end{prop}

This bijection links hypertrees with $k$ edges with hypertrees with $k$ edges, and therefore it preserves the weight on hypertrees. Pointing a vertex or an edge of a hypertree and then decorating its edges is just the same as decorating its edges and then pointing a vertex or an edge. Therefore, the decoration of edges is compatible with the previous property \ref{prindi} and we obtain: 

\begin{prop}[Dissymmetry principle for edge-decorated hypertrees] \label{principe de dissymétrie} Given a species $\mathcal{S}$, the following relation holds: 
\begin{equation}
\hs+\hspa=\hsp+\hsa. \label{relpd}
\end{equation}
\end{prop}

		\subsubsection{Functional equations}

To establish relations between species, we will need the following operations on species: 
		
		\begin{defi} Let $F$ and $G$ be two species. We define the following operations on species: 
\begin{itemize}
\item $\left(F + G \right)\left(I\right)=F\left(I\right) \sqcup G\left(I\right)$, (addition)
\item $\left(F\times G \right)\left(I\right)=F\left(I\right) \times G\left(I\right)$, (product)
\item $\left(F \circ G\right)\left(I\right)=\bigsqcup_{\pi \in \mathcal{P}\left(I\right)} F\left(\pi\right) \times \prod_{J \in \pi} G\left(J\right) $, where $\mathcal{P}\left(I\right)$ runs on the set of partitions of $I$.(plethystic substitution)
\end{itemize}
\end{defi}
		
The previous species are related by the following proposition:

\begin{prop} \label{décomp} Given a species or a linear species $\mathcal{S}$ such that $\lvert S\left(\emptyset\right) \rvert = 0$ and $\lvert S\left(\{1\}\right) \rvert= 0 $, the species $\hs$, $\hsp$, $\hsa$, $\hspa$ and $\hsc$ satisfy: 
\begin{equation} 
\hsp =X \times \hsprime,
\end{equation}

\begin{equation} \label{relhp}
t \hsp = X + X \times \comm \circ\left(t \times \hsc\right),
\end{equation}

\begin{equation} \label{relhc}
\hsc = S' \circ t \hsp,
\end{equation}

\begin{equation} \label{rela}
\hsa= S \circ t \hsp,
\end{equation}

\begin{equation} \label{relpa}
\hspa =\hsc \times t \hsp = X \times \bigl( X\left(1+\comm\right)\bigr) \circ \hsc.
\end{equation}
\end{prop}

\begin{proof}
If we multiply the series by $t$, the power of $t$ corresponds to the number of edges in the associated hypertrees.
\begin{itemize}
\item The first relation is due to the relation between a species and its rooted variants.

\item The second one is obtained from a decomposition of a rooted hypertree. If there is only one vertex, the label of the hypertree corresponds to the singleton species $X$. As there is no edge, the power of $t$ is $0$. Otherwise, we separate the label of the root which correxponds to a multiplication by $X$. It remains a hypertree with a gap contained in different edges. Separating these edges, we obtain a non-empty set of hollow hypertrees with edges decorated by $\mathcal{S}$. There is the same number of edges in the set of hollow hypertrees as in the rooted hypertree. This operation is a bijection of species because a vertex alone is a rooted edge-decorated hypertree and taking a non-empty forest of hollow edge-decorated hypertrees and linking them by their gap on which we put a label gives a rooted edge-decorated hypertree.

\item The third relation is obtained by pointing the vertices in the hollow edge and breaking the edge: we obtain a non-empty forest of rooted edge-decorated hypertrees. As we break an edge, there is one edge less in the forest of rooted hypertrees than in the hollow hypertree: we can simplify the $t$ in front of $\hsc$. The set of roots is a $\mathcal{S}'$-structure and induces this structure on the set of trees: we obtain a $\mathcal{S}'$-structure in which all elements are rooted edge-decorated hypertrees. As this operation is reversible and does not depend on the labels of the hollow hypertree, this is a bijection of species.

\item The fourth relation is obtained by pointing the vertices in the pointed edge and breaking it: we obtain a non-empty forest of at least two rooted edge-decorated hypertrees. As we break an edge, there is one edge less in the forest of rooted hypertrees than in the edge-pointed hypertree: we can simplify the $t$ in front of $\hsa$. The set of roots is a $\mathcal{S}$-structure and induces this structure on the set of tree: we obtain a $\mathcal{S}$-structure in which all elements are rooted edge-decorated hypertrees. As this operation is reversible and does not depend on the labels of the hollow hypertree, this is a bijection of species.

\item The last relation is obtained by separating the pointed edge from the other edges containing the root and putting a gap in the pointed edge instead of the root: the connected component of the pointed edge gives an edge-decorated hollow hypertree and the connected component containing the root gives a rooted edge-decorated hypertree. There is the same number of edges in the edge-pointed rooted hypertree as in the union of the hollow and the rooted hypertree. \qedhere
\end{itemize} 
\end{proof}

\begin{cor} \label{coreq}
Using the equations \eqref{relhp} and \eqref{relhc} of the previous proposition, we obtain: 
\begin{equation}
t \hsp = X +X \times \comm \circ \left( t \times S' \circ t \hsp \right)
\end{equation}
and
\begin{equation}
\hsc = S' \circ \left( X + X \times \comm \circ\left(t \times \hsc\right) \right).
\end{equation}
\end{cor}

\section{Counting edge-decorated hypertrees using box trees}

In this section, we count edge-decorated hypertrees and rooted and pointed variants of them using a new type of tree-like structure, called box trees.

	\subsection{Box trees} 
	
Let us consider a quadruple $\left(L,V,R,E\right)$, where
\begin{itemize}
\item $L$ is a finite set called the set of \emph{labels},
\item $V$ is a partition of $L$ called the set of \emph{vertices},
\item $R$ is an element of $V$ called the root,
\item $E$ is a map from $V-\{R\}$ to $L$ called the set of \emph{edges}. 
\end{itemize} 
We will denote by $\tilde{E}$, the map from $V-\{R\}$ to $V$ which associates to a vertex $v$ the vertex $v'$ containing the label $E\left(v\right)$. The pair $\left(V,\tilde{E}\right)$ is then an oriented graph, with vertices labelled by subsets of $L$.

\begin{defi}
A quadruple $\left(L,V,R,E\right)$ is a \emph{box tree} if and only if the graph $\left(V,\tilde{E}\right)$ is a tree, rooted in $R$, with edges oriented toward the root.

A label $l$ is called \emph{parent} of a vertex $v$ if $E\left(v\right)=l$.
\end{defi}	

In figure \ref{exemple box tree}, an example of box trees is presented. The root is the double rectangle.

\begin{figure} 
\centering 
\begin{tikzpicture}[scale=1]
\draw[black, rounded corners] (-0.35,-0.35) rectangle(2.35,0.35);
\draw[black, rounded corners] (-0.3,-0.3) rectangle(2.3,0.3);
\draw[black, rounded corners] (-0.3,0.7) rectangle(1.3,1.3);
\draw[black, rounded corners] (1.7,0.7) rectangle(2.3,1.3);
\draw[black, rounded corners] (-0.3,1.7) rectangle(0.3,2.3);
\draw[black, rounded corners] (0.7,1.7) rectangle(1.3,2.3);
\draw[black, rounded corners] (1.7,1.7) rectangle(2.3,2.3);
\draw (1,0.7) -- (1,0.15);
\draw (2,0.7) -- (2,0.15);
\draw (1,1.7) -- (1,1.15);
\draw (0,1.7) -- (1,1.15);
\draw (2,1.7) -- (2,1.15);
\draw (0,0) node{$1$};
\draw (1,0) node{$2$};
\draw (2,0) node{$3$};
\draw (2,1) node{$4$};
\draw (2,2) node{$5$};
\draw (0,1) node{$6$};
\draw (1,1) node{$7$};
\draw (1,2) node{$8$};
\draw (0,2) node{$9$};
\end{tikzpicture}
\caption{\label{exemple box tree} A box tree. }
\end{figure}
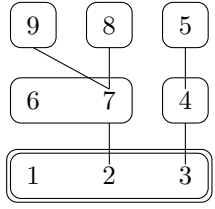

The difference between box trees and fat trees is mainly in the edges, which are between labels for fat trees and a label and a vertex for box trees.

\begin{prop} \label{enum box tree}
Let us consider $L$ a finite set of cardinality $n$ and $V$ a partition of $L$ into $k+1$ parts $p_0,p_1, \ldots, p_k$, where the cardinality of $p_i$ is $n_i$. 
The number of box trees $N_{p_0;p_1, \ldots, p_k}$ on $k+1$ vertices with root labelled by $p_0$ and the other $k$ vertices labelled by one of the $k$ other $p_i$ is given by: 
\begin{equation*}
N_{p_0;p_1, \ldots, p_k}=n_0 \times n^{k-1}.
\end{equation*}
\end{prop}

\begin{proof}
We prove this statement by induction on $k$.

For $k=1$, there is only one vertex attached to a label of the root. As there is $n_0$ labels in the root, there is $n_0$ box trees on two vertices.

If this statement holds for all $q<k$, we compute the number $N_{p_0;p_1, \ldots, p_k}$ of box tree on $k+1$ vertices satisfying the hypothesis. It can be obtained by summing on the number of vertices attached to the root, called its \emph{children}.

If the root has $j$ children, with respectively $n_{i_1}, \ldots, n_{i_j}$ labels, there are $n_0^j$ ways of attaching them to the root. Moreover, cutting the root and gluing together the children of the root in one vertex, we obtain a box tree on $k+1-j$ vertices and $n-n_0$ labels, with root having $n_{i_1} + \cdots + n_{i_j}$ labels. Using the induction hypothesis, we obtain: 

\begin{equation*}
N_{p_0;p_1, \ldots, p_k } = \sum_{j=1}^k n_0^j \sum_{0<i_1 < \cdots < i_j}\left(n_{i_1} + \cdots + n_{i_j}\right) \left(n-n_0\right)^{k-j-1}.
\end{equation*}

In the second sum $\sum_{0<i_1 < \cdots < i_j}\left(n_{i_1} + \cdots + n_{i_j}\right)$, every $n_i$ appears $\binom{k-1}{j-1}$ times, for $i \geq 1$. Indeed, in this case the vertex labelled by $p_i$ is a child of the root so we choose the $j-1$ other children among the $k-1$ other vertices. Therefore, we have: 

\begin{equation*}
N_{p_0;p_1, \ldots, p_k } = \sum_{j=1}^k \binom{k-1}{j-1} n_0^j\left(n-n_0\right)^{k-j} = n_0 \times n^{k-1}.
\end{equation*}

This gives the expected result.
\end{proof}

	\subsection{Enumeration of decorated hypertrees}

\begin{prop} \label{btrhypertree}
Given a species $\mathcal{S}$, every rooted hypertree with $k$ edges and $n$ vertices, whose edges are decorated by $\mathcal{S}$, can be decomposed as a triple $\left(r, \mathbb{S}, BT\right)$ where: 
\begin{itemize}
\item $r$ is the root of the hypertree, 
\item $\mathbb{S}$ is a set of $k$ sets on $n-1$ vertices with a $\mathcal{S}'$-structure on each of them,
\item and $BT$ is a box tree on $k+1$ vertices with root labelled by $r$ and each of the other vertices is labelled by one of the $k$ sets of the previous point.
\end{itemize}
\end{prop}

\begin{proof}

Given a rooted edge-decorated hypertree $H$ with $k$ edges and $n$ vertices and with root labelled by $r$, the edges give a set $\mathbb{S}$ of $k$ sets on $n-1$ vertices with a $\mathcal{S}'$-structure on each of them. Indeed, considering an edge $e$, the set of vertices of $e$ different from the petiole is endowed with a $\mathcal{S}'$-structure, because the rooted hypertree is edge-decorated. Moreover, every vertex different from the root is the petiole of all edges containing it, except the closest edge from the root, which always exists as the hypertree is connected. Therefore, every vertex different from the root is counted once in the sets of $\mathbb{S}$. Then, we call  \emph{set of vertices} $V$ the set of the $k$ previous sets and the root. For all edges $e$ in $H$, we can link the corresponding vertex of $V$ to the petiole of $e$ which belongs to another vertex of $V$: let us call $BT$ the result. As $H$ is a hypertree, for every vertex $v$ there is only one path between the root and $v$ thus, there is also one and only one path from the root to the vertices containing $v$ in $BT$: $BT$ is a box tree with root labelled by $r$. 

Conversely, given a box tree $BT$ with root labelled by $r$ and the other vertices labelled by one of the $k$ sets on $n-1$ vertices with a $\mathcal{S}'$-structure on each of them, we can call \emph{parent} of a vertex the label linked to it with an edge. Calling the set of labels $V$, we can define $E$ as a set of subsets of $V$ obtained by taking, for every vertex of $BT$, the union of the set of its labels with its parent. We thus obtain a hypertree rooted in $r$. Moreover, the $\mathcal{S}$-structure on every vertex of $BT$ induces an $\mathcal{S}$-edge-decoration of the rooted hypertree.

Let us remark that this application is a bijection of sets but not a bijection of species.
\end{proof}

\begin{exple} \label{exple box rooted}
The previous bijection associates to the edge-decorated rooted hypertree of figure \ref{exemple rooted} the root $3$, the set of lists $\{\left(15,1\right)$, $\left(14,7,13,2\right)$, $\left(4,5,6\right)$, $\left(8\right)$, $\left(10,9\right)$, $\left(11\right)$, $\left(12\right)\}$ and the box tree in figure \ref{btf}.

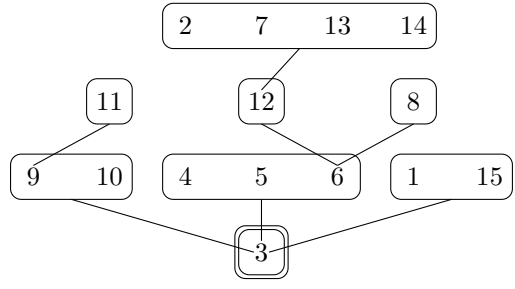
\begin{figure} 
\centering 
\begin{tikzpicture}[scale=1]
\draw[black, rounded corners] (2.7,-0.3) rectangle(3.3,0.3);
\draw[black, rounded corners] (2.65,-0.35) rectangle(3.35,0.35);
\draw[black, rounded corners] (-0.3,0.7) rectangle(1.3,1.3);
\draw[black, rounded corners] (1.7,0.7) rectangle(4.3,1.3);
\draw[black, rounded corners] (4.7,0.7) rectangle(6.3,1.3);
\draw[black, rounded corners] (0.7,1.7) rectangle(1.3,2.3);
\draw[black, rounded corners] (4.7,1.7) rectangle(5.3,2.3);
\draw[black, rounded corners] (2.7,1.7) rectangle(3.3,2.3);
\draw[black, rounded corners] (1.7,2.7) rectangle(5.3,3.3);
\draw (0.5,0.7) -- (2.9,0);
\draw (3,0.7) -- (3,0.15);
\draw (5.5,0.7) -- (3.1,0);
\draw (1,1.7) -- (0,1.15);
\draw (3,1.7) -- (4,1.15);
\draw (5,1.7) -- (4,1.15);
\draw (3,2.15) -- (3.5,2.7);
\draw (3,0) node{$3$};
\draw (0,1) node{$9$};
\draw (1,1) node{$10$};
\draw (2,1) node{$4$};
\draw (3,1) node{$5$};
\draw (4,1) node{$6$};
\draw (5,1) node{$1$};
\draw (6,1) node{$15$};
\draw (1,2) node{$11$};
\draw (5,2) node{$8$};
\draw (3,2) node{$12$};
\draw (2,3) node{$2$};
\draw (3,3) node{$7$};
\draw (4,3) node{$13$};
\draw (5,3) node{$14$};
\end{tikzpicture}
\caption{The box tree associated with the edge-decorated rooted hypertree of figure \ref{exemple rooted}.\label{btf}}
\end{figure}

The order of the elements in each box does not matter.
\end{exple}

\begin{prop} \label{bthhypertree} Given a species $\mathcal{S}$, every hollow hypertree with $k$ edges and $n$ vertices, whose edges are decorated by $\mathcal{S}$ can be decomposed as a pair $\left(\mathbb{S}, BT\right)$ where: 
\begin{itemize}
\item $ \mathbb{S}$ is a set of $k$ sets on $n$ vertices with a $\mathcal{S}'$-structure on each of them,
\item and $BT$ is a box tree on $k$ vertices with each vertex labelled by one of the $k$ sets of the previous point.
\end{itemize}
\end{prop}

\begin{proof}
We adapt the proof of the theorem for rooted hypertrees to hollow hypertrees.

Given a hollow edge-decorated hypertree $H$ with $k$ edges and $n$ vertices, the edges give a set $ \mathbb{S}$ of $k$ sets on $n$ vertices with a $\mathcal{S}'$-structure on each of them. Indeed, considering an edge $e$, the set of vertices of $e$ different from the petiole or the gap of $e$ is endowed with a $\mathcal{S}'$-structure, because the hollow hypertree is edge-decorated. Moreover, every vertex is the petiole of all edges containing it, except the closest from the hollow edge, which always exists as the hypertree is connected. Therefore every vertex is counted exactly once in the set $ \mathbb{S}$.Then, we call  \emph{set of vertices} $V$ the set of the $k$ previous sets. For all non-hollow edges $e$ in $H$, we can link the corresponding vertex of $V$ to the petiole of $e$ which belongs to another vertex of $V$: let us call $BT$ the result. The vertices that were in the hollow edge form the root of $BT$. As $H$ is a hypertree, for every vertex $v$ there is only one path between the hollow edge and $v$ thus, there is also one and only one path from the root to the vertices containing $v$ in $BT$: $BT$ is a box tree with root labelled by vertices of the hollow edge.

Conversely, given a box tree $BT$ with vertices labelled by one of the $k$ sets on $n$ vertices with a $\mathcal{S}'$-structure on each of them, we can call \emph{parent} of a vertex the label linked to it with an edge. Calling the set of labels $V$, we can define $E$ as a set of subsets of $V$ obtained by taking, for every vertex of $BT$ different from its root, the union of the set of its labels with its parent and by taking the set of labels of the root of $BT$ with a gap. We thus obtain a hollow hypertree. Moreover, the $\mathcal{S}'$-structure on every vertex of $BT$ induces an $\mathcal{S}$-edge-decoration of the hollow hypertree.
\end{proof}

\begin{exple} \label{exple box hollow} Figure \ref{exemple box tree} is the box tree associated to the hypertree of figure \ref{exemple hollow}, with the set of pointed sets $\{ \{\mathbf{1},2,3\}$,$\{6,\mathbf{7}\}$, $\{\mathbf{8}\} $, $\{\mathbf{9}\} $, $\{\mathbf{4}\} $, $\{\mathbf{5}\} \}$.

\end{exple}

To every species $F$, we can associate the following generating series: 
\begin{equation*}
C_F\left(x\right)=\sum_{n \geq 0} \lvert F\left(\{1,\ldots, n\}\right) \rvert \frac{x^n}{n!}. 
\end{equation*}

We can combine this series with the weight $t$.

\begin{exple} The generating series of species defined in Example \ref{exple espèces} are: 
\begin{itemize}
\item $C_{\assoc}\left(x\right)=\frac{1}{1-x}$,
\item $C_\mathcal{E}\left(x\right)=\exp\left(x\right)$,
\item $C_X\left(x\right)=x$,
\item $C_{\operatorname{Comm}}\left(x\right)=\exp\left(x\right)-1$.
\end{itemize}
\end{exple}

Let $\Ss$, $\Sp$, $\Sc$, $\Sa$ and $\Spa$ be the weighted generating series of $\hs$, $\hsp$, $\hsc$, $\hsa$ and $\hspa$. Let also $E_\mathcal{S}\left(k,n\right)$ be the number of sets of $k$ sets on $n$ vertices with a $\mathcal{S}'$-structure on each of them. By convention, we choose $E_\mathcal{S}\left(1,1\right)$ to be equal to $1$.

Using Proposition \ref{enum box tree}, the previous propositions \ref{btrhypertree} and \ref{bthhypertree} imply: 

\begin{thm} \label{cor S}
Given a species $\mathcal{S}$, the generating series of the species of edge-decorated rooted hypertrees, edge-decorated hollow hypertrees and edge-decorated hypertrees have the following expressions: 
\begin{equation} \label{formule Sp}
\Sp\left(x\right)= \frac{x}{t} + \frac{1}{t} \sum_{n \geq 2} \sum_{k=1}^{n-1} E_\mathcal{S}\left(k,n-1\right)\left( nt \right)^{k} \frac{x^n}{n!},
\end{equation}
\begin{equation}
\Sc\left(x\right)= \sum_{n \geq 1} \sum_{k=1}^{n} E_\mathcal{S}\left(k,n\right)\left(nt\right)^{k-1} \frac{x^n}{n!}
\end{equation}
and
\begin{equation} \label{formule S}
\Ss\left(x\right)= \frac{x}{t} + \sum_{n \geq 2} \sum_{k=1}^{n-1} E_\mathcal{S}\left(k,n-1\right)\left(nt\right)^{k-1} \frac{x^n}{n!},
\end{equation}
where $E_\mathcal{S}\left(k,n\right)$ is the number of sets of $k$ sets on $n$ vertices with a $\mathcal{S}'$-structure on each of them.

The generating series of the species of rooted edge-decorated edge-pointed hypertrees and the species of edge-decorated edge-pointed hypertrees are related to the species of rooted edge-decorated hypertrees and the species of hollow edge-decorated hypertrees by the following relations: 
\begin{equation}
\Spa\left(x\right)= t \Sp\left(x\right)\times \Sc\left(x\right),
\end{equation}

\begin{equation}
\Sa\left(x\right)= \Ss\left(x\right)+\Spa\left(x\right)-\Sp\left(x\right).
\end{equation}

These formulas give the cardinality of all types of decorated hypertrees in terms of decorated sets.
\end{thm}

\begin{proof}
\begin{itemize}
\item According to Proposition \ref{btrhypertree}, the number of rooted hypertrees on $k$ edges and $n$ vertices is given by: 
\begin{equation*}
\left(\Sp\right)_{n,k}\left(x\right)= n \times E_\mathcal{S}\left(k,n-1\right) n^{k-1} t^{k-1},
\end{equation*}
where we have $n$ different ways to choose the root, $E_\mathcal{S}\left(k,n-1\right)$ ways to make a set of $k$ sets on the $n-1$ vertices left with a $\mathcal{S}'$-structure on each of them and $1 \times n^{k-1}$ ways to organize these sets into a box tree with its root fixed, according to Proposition \ref{enum box tree}.

\item Let us apply Proposition \ref{bthhypertree}. We have $E_\mathcal{S}\left(k,n\right)$ ways to make a set of $k$ sets on the $n$ vertices with a $\mathcal{S}'$-structure on each of them. Then consider that $n_{i_1}, \ldots, n_{i_k}$ are the cardinality of each of these sets. We choose the $j$-th of these sets as the root of a box tree on $k$ vertices and $n$ labels. The number of box trees obtained is $n_{i_j} \times n^{k-2}$, according to Proposition \ref{enum box tree}.
As $\sum_{j=1}^k n_{i_j} \times n^{k-2} = n^{k-1}$, we obtain the expected result.
\item We prove Equation \eqref{formule S} by an integration of Formula \eqref{formule Sp}. 
\item The last equations are obtained by translating in terms of generating series Equations \eqref{relpd} and \eqref{relpa} on species.
\qedhere
\end{itemize}
\end{proof}

\begin{exple} \begin{itemize}
\item Let $\assoc$ be the species of non-empty lists. The number of partitions of a set of cardinality $n$ in $k$ lists is $\binom{n-1}{k-1}\frac{n!}{k!}$.
Therefore, if we consider hypertrees decorated by $\widehat{\operatorname{Assoc}}$, the generating series of edge-decorated rooted hypertrees and edge-decorated hollow hypertrees are: 
\begin{equation*}
\Sp\left(x\right)= \frac{x}{t}+ \frac{1}{t} \sum_{n \geq 2} \sum_{k=1}^{n-1} \binom{n-2}{k-1} \frac{\left(n-1\right)!}{k!}\left(nt\right)^{k} \frac{x^n}{n!}
\end{equation*}
and
\begin{equation*}
\Sc\left(x\right)= \sum_{n \geq 1} \sum_{k=1}^{n} \binom{n-1}{k-1} \frac{n!}{k!}\left(nt\right)^{k-1} \frac{x^n}{n!}.
\end{equation*}

\item Let $\perm$ be the species of non-empty pointed sets. The number of partitions of a set of cardinality $n$ in $k$ pointed sets is $\binom{n}{k} \times k^{n-k}$. Indeed, we choose the pointed vertex in each set and then we map the other vertices to these pointed vertices.
Therefore, if we consider hypertrees decorated by $\widehat{\operatorname{Perm}}$, the generating series of edge-decorated rooted hypertrees and edge-decorated hollow hypertrees are: 
\begin{equation*}
\Sp\left(x\right)= \frac{x}{t} + \frac{1}{t} \sum_{n \geq 2} \sum_{k=1}^{n-1} \binom{n-1}{k} k^{n-1-k}\left(nt\right)^{k} \frac{x^n}{n!}
\end{equation*}
and
\begin{equation*}
\Sc\left(x\right)=\sum_{n \geq 1} \sum_{k=1}^{n} \binom{n}{k} k^{n-k}\left(nt\right)^{k-1} \frac{x^n}{n!}.
\end{equation*}

\item Let $\comm$ be the species of non-empty sets. The number of partitions of a set of cardinality $n$ in $k$ sets is given by $S\left(n,k\right)$, a Stirling number of the second type. 
Therefore, if we consider hypertrees decorated by $\widehat{\operatorname{Comm}}$, the generating series of edge-decorated rooted hypertrees and edge-decorated hollow hypertrees are: 
\begin{equation*}
\Sp\left(x\right)= \frac{x}{t} + \frac{1}{t} \sum_{n \geq 2} \sum_{k=1}^{n-1} S\left(n-1,k\right)\left(nt\right)^{k} \frac{x^n}{n!},
\end{equation*}
and
\begin{equation*}
\Sc\left(x\right)= \sum_{n \geq 1} \sum_{k=1}^{n} S\left(n,k\right)\left(nt\right)^{k-1} \frac{x^n}{n!}.
\end{equation*}
These decorated hypertrees are isomorphic to non-decorated hypertrees. 
This result was first proven by I. Gessel and L. Kalikow in \cite{GesKal}.

\item The number of partitions of a set of cardinality $n$ in $k$ cycles is given by $|s\left(n,k\right)|$, the absolute value of a Stirling number of the first kind. 
Therefore, if we consider hypertrees decorated by the species of cycles, the generating series of edge-decorated rooted hypertrees and edge-decorated hollow hypertrees are: 
\begin{equation*}
\Sp\left(x\right)= \frac{x}{t}+ \frac{1}{t} \sum_{n \geq 2} \sum_{k=1}^{n-1} |s\left(n-1,k\right)|\left(nt\right)^{k} \frac{x^n}{n!}
\end{equation*}
and
\begin{equation*}
\Sc\left(x\right)= \sum_{n \geq 1} \sum_{k=1}^{n} |s\left(n,k\right)|\left(nt\right)^{k-1} \frac{x^n}{n!}.
\end{equation*}

These series are the same as the one for the decoration by $\widehat{\operatorname{Lie}}$. We will see this last case in details in Section \ref{Lie}.
\end{itemize}
\end{exple}

	\subsection{Refinement} 
	
	\label{refinement}

Let us now introduce two weights: one on the set of box trees and the other on the set of rooted and hollow hypertrees.

\begin{defi} \label{WBT}
Let $BT=\left(L,V,R,E\right)$ be a box tree, we define the following weight on it: 
\begin{equation*}
W\left(BT\right)=\prod_{i \in L} x_i^{\lvert E^{-1}\left(i\right) \rvert},
\end{equation*}
where the $x_i$ are formal variables.
The power of $x_i$ is then the number of children of the label $i$, for every label $i$ of $L$.

\end{defi}

With this weight, the number of box trees is given by: 

\begin{prop} \label{enumx box tree}
The number of weighted box trees $N_{\{i_1, \ldots, i_p\};k,n}$ with $n$ labels $\{1, \ldots, n\}$, $k+1$ vertices, $k$ edges and its root labelled by $\{i_1, \ldots, i_p\}$ is: 
\begin{equation*}
N_{\{i_1, \ldots, i_p\};k,n}=\left(x_{i_1} + \cdots + x_{i_p}\right) \left( \sum_{i=1}^n x_i \right)^{k-1}.
\end{equation*}
\end{prop}

\begin{proof}
We prove this statement by induction on $k$. Let us call $V_1$ the root and $V_2, \ldots, V_k$ the other vertices.

For $k=1$, there is only one vertex attached to a label of the root. Thus, the weight of box trees on two vertices with root $V_1$ is $\sum_{i \in V_1} x_i$.

If this statement holds for all $q<k$, we want the number $N_{\{i_1, \ldots, i_p\};k,n}$ of box trees on $k+1$ vertices. It can be obtained by summing on the number of vertices attached to the root, called its \emph{children}.

If the root has $j$ children $V_{a_1}, \ldots, V_{a_j}$, each of them is attached to a label of the root: this gives a term $\left(x_{i_1}+ \cdots + x_{i_p}\right)^j$. Moreover, cutting the root and linking together the children of the root, we obtain a box tree on $k+1-j$ vertices and with the same labels except the ones of $V_0$ which is deleted, with root having the labels of $V_{a_1} \cup \ldots \cup V_{a_j}$. Using the induction hypothesis, we obtain: 

\begin{equation*}
N_{\{i_1, \ldots, i_p\};k,n} = \sum_{j=1}^{k-1}\left(x_{i_1}+ \cdots + x_{i_p}\right)^j \sum_{0 <i_1 < \cdots < i_j} \left( \sum_{i \in V_{i_1} \cup \ldots \cup V_{i_j}} x_i \right) \left( \sum_{i \notin V_0} x_i \right)^{k-j-1}.
\end{equation*}

In the second sum $ \sum_{0 <i_1 < \cdots < i_j}\left(\sum_{i \in V_{i_1} \cup \ldots \cup V_{i_j}} x_i\right)$, every $x_\alpha \in V_i$ appears $\binom{k-1}{j-1}$ times, for $i \geq 2$. Therefore, we have: 

\begin{align*}
N_{\{i_1, \ldots, i_p\};k,n} &= \sum_{j=1}^{k-1} \binom{k-1}{j-1}\left(\sum_{i \in V_0} x_i\right)^j \left(\sum_{i \notin V_0} x_i\right)^{k-j} \\ & =\left(\sum_{i \in V_0} x_i\right) \left(\sum_{l=0}^k \sum_{i \in V_l} x_i\right)^{k-1}. \qedhere
\end{align*}
\end{proof}

The weight on the set of hypertrees on $n$ vertices is defined as follow.

\begin{defi}
Let $H=\left(V,E\right)$ be a rooted or a hollow hypertree on $n$ vertices, we define the following weight on it: 
\begin{equation*}
W\left(H\right)=\prod_{i \in V} x_i^{p\left(i\right)},
\end{equation*}
where $p\left(i\right)$ is the number of edges whose petiole is $i$.
\end{defi}

These weights are related by the following proposition: 

\begin{thm} \label{decomp root}
Given a species $\mathcal{S}$, every weighted rooted hypertree on $n$ vertices with $k$ edges, whose edges are decorated by $\mathcal{S}$ can be decomposed as a triple $\left(r, \mathbb{S}, BT\right)$ where: 
\begin{itemize}
\item $r$ is the root of the hypertree, 
\item $ \mathbb{S}$ is a set of $k$ sets on $n-1$ vertices with a $\mathcal{S}'$-structure on each of them,
\item and $BT$ is a weighted box tree on $k+1$ vertices with root labelled by $r$ and each of the other vertices labelled by one of the $k$ sets of the previous point.
\end{itemize}
\end{thm}

\begin{proof}
We only have to prove the compatibility with weights of the bijection of the proof of Proposition \ref{btrhypertree}. This compatibility is due to the fact that the set of petioles of both rooted hypertree and box tree related in the theorem is the same.
\end{proof}

\begin{exple}
The weight of the rooted hypertree of Example \ref{exemple rooted}, which corresponds to the weight of the box tree of Example \ref{exple box rooted}, is: $x_3^3 x_6^2 x_9 x_{12}$.
\end{exple}

\begin{thm} \label{decomp hol}
Given a species $\mathcal{S}$, every weighted hollow hypertree on $n$ vertices with $k$ edges, whose edges are decorated by $\mathcal{S}$ can be decomposed as a pair $\left(\mathbb{S}, BT\right)$ where: 
\begin{itemize}
\item $ \mathbb{S}$ is a set of $k$ sets on $n$ vertices with a $\mathcal{S}'$-structure on each of them,
\item and $BT$ is a weighted box tree on $k$ vertices with each vertex labelled by one of the $k$ sets of the previous point.
\end{itemize}
\end{thm}

\begin{proof}
In the same way as we proved the compatibility with weights of the bijection of the proof of Proposition \ref{btrhypertree} in the case of rooted weighted hypertrees, we can prove that the proof of Proposition \ref{bthhypertree} is compatible with weights.
\end{proof}

\begin{exple}
The weight of the rooted hypertree of figure \ref{exemple hollow}, which corresponds to the weight of the box tree of figure \ref{exemple box tree}, is: $x_2 x_3 x_4 x_7^2$.
\end{exple}

Let $\Ssw$, $\Spw$, $\Scw$, $\Saw$ and $\Spaw$ be the weighted generating series of $\hs$, $\hsp$, $\hsc$, $\hsa$ and $\hspa$. Let also be $E_\mathcal{S}\left(k,n\right)$ the number of sets of $k$ sets on $n$ vertices with a $\mathcal{S}'$-structure on each of them.

Using Proposition \ref{enumx box tree}, the previous theorems \ref{decomp hol} and \ref{decomp root} imply: 

\begin{cor} \label{formule Spx}
The generating series of the species of edge-decorated rooted hypertrees and edge-decorated hollow hypertrees have the following expressions: 
\begin{equation}
\Sp\left(x\right)= \frac{x}{t} + \sum_{n \geq 2} \sum_{k=1}^{n-1} \left(x_1 + \cdots+ x_n\right) E_\mathcal{S}\left(k,n-1\right) \bigl( \left(x_1 + \cdots+ x_n\right)t \bigr)^{k-1} \frac{x^n}{n!}
\end{equation}
and
\begin{equation}
\Sc\left(x\right)= \sum_{n \geq 1} \sum_{k=1}^{n} E_\mathcal{S}\left(k,n\right) \bigl(\left(x_1 + \cdots + x_n\right) t \bigr)^{k-1} \frac{x^n}{n!}.
\end{equation}
\end{cor}

This corollary will be used in the next section where we specialize the results for operads $\lie$ and $\prelie$.

\section{Two cases of edge-decorated hypertrees}

	\subsection{$\widehat{\prelie}$-decorated hypertrees}
	
	In their article \cite{JMcCM}, C. Jensen, J. McCammond and J. Meier introduce a weight on the set of hypertrees. We prove here that this weight is related to a decoration of edges by the linear species $\widehat{\prelie}$, whose differential is $\prelie$, which associates to a finite set $I$ the set of labelled rooted trees with labels in $I$.	
	
		\subsubsection{Application of the enumeration of decorated hypertrees}	
	
	Applying the results of the previous section on counting edge-decorated hypertrees, we obtain the following proposition: 
	
	\begin{prop} \label{sprelie}
	The generating series of the species of rooted hypertrees decorated by the linear species $\widehat{\prelie}$ is given by: 
	\begin{equation*}
S^r_{\widehat{\prelie}}= \frac{x}{t} +\sum_{n \geq 2}	n\left(tn+n-1\right)^{n-2} \frac{x^n}{n!}.
	\end{equation*}
	This equation is the specialization in $x_i=1$ of the one counting rooted weighted hypertrees in \cite[Theorem 3.9]{JMcCM}.
	
	The generating series of the species of hypertrees decorated by $\widehat{\prelie}$ is given by: 
	\begin{equation*}
S_{\widehat{\prelie}}=x+\sum_{n \geq 2}\left(tn+n-1\right)^{n-2} \frac{x^n}{n!}.
	\end{equation*}
	
		The generating series of the species of hollow hypertrees decorated by $\widehat{\prelie}$ is given by: 
	\begin{equation*}
S^h_{\widehat{\prelie}}=\sum_{n \geq 1}\left(tn+n\right)^{n-1} \frac{x^n}{n!}.
	\end{equation*}
	
	\end{prop}
	
	\begin{proof}
	
There is a classical formula for the number of forest of $k$ trees on $n$ vertices, which can be found in the book of M. Aigner and G. Ziegler \cite{PftB}: 
\begin{equation} \label{eprelie}
E_{\prelie}\left(k,n\right)=\binom{n}{k} k \times n^{n-1-k}= \binom{n-1}{k-1} n^{n-k}.
\end{equation}
	
	Using the expressions of Theorem \ref{cor S}, we obtain: 
\begin{equation*}
S^r_{\widehat{\prelie}}\left(x\right)= \frac{x}{t} + \sum_{n \geq 2} \sum_{k=1}^{n-1} n \binom{n-2}{k-1} \left(n-1\right)^{n-1 -k} \left(nt\right)^{k-1} \frac{x^n}{n!}
\end{equation*}
and
\begin{equation*}
S^h_{\widehat{\prelie}}\left(x\right)= \sum_{n \geq 1} \sum_{k=1}^{n} \binom{n-1}{k-1} n^{n-1-\left(k-1\right)} \left(nt\right)^{k-1} \frac{x^n}{n!}.
\end{equation*}

Re-indexing the sums, it gives: 
\begin{equation*}
S^r_{\widehat{\prelie}}= \frac{x}{t} + \sum_{n \geq 2} \sum_{k=0}^{n-2} n \binom{n-2}{k} \left(n-1\right)^{n-2-k}\left(nt\right)^{k} \frac{x^n}{n!}
\end{equation*}
and
\begin{equation*}
S^h_{\widehat{\prelie}}\left(x\right)= \sum_{n \geq 1} \sum_{k=0}^{n-1} \binom{n-1}{k} n^{n-1-k} \left(nt\right)^{k} \frac{x^n}{n!}.
\end{equation*}
	
Using the binomial theorem, we obtain the expected results for $S^r_{\widehat{\prelie}}$ and $S^h_{\widehat{\prelie}}$. The series $S_{\widehat{\prelie}}$ is then obtained by the use of the first equation of Proposition \ref{décomp}.
	\end{proof}
	
		\subsubsection{Link with 2-coloured rooted trees}		

We now draw the link between trees whose edges can be either blue ($0$) or red ($1$) and edge-decorated hypertrees to compute the generating series of edge-pointed and rooted edge-pointed decorated hypertrees.
	
	\begin{defi}
	A \emph{2-coloured rooted tree} is a rooted tree $\left(V,E\right)$, where $V$ is the set of vertices and $E \subseteq V \times V$ is the set of edges, together with a map $\varphi$ from $E$ to $\{0,1\}$. It is equivalent to the data of a tree $\left(V,E\right)$ and a partition $E_0 \cup E_1$ of $E$, with $E_0 \cap E_1 = \emptyset$.
	\end{defi}

	An example of 2-coloured rooted tree is presented in figure \ref{2-col}. The edges of $E_1$ are dashed whereas the edges of $E_0$ are plain.
	
	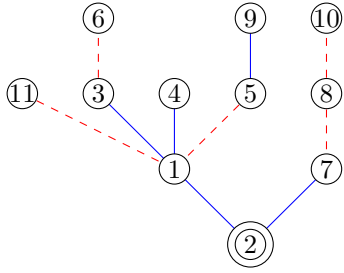
\begin{figure} 
	\centering 
	\begin{tikzpicture}[scale=1]
\draw[blue] (3,0) -- (2,1);
\draw[red, dashed] (0,2) -- (2,1);
\draw[blue] (1,2) -- (2,1);
\draw[red, dashed] (1,2) -- (1,3);
\draw[blue] (2,2) -- (2,1);
\draw[red, dashed] (3,2) -- (2,1);
\draw[blue] (3,2) -- (3,3);
\draw[blue] (3,0) -- (4,1);
\draw[red, dashed] (4,2) -- (4,1);
\draw[red, dashed] (4,2) -- (4,3);
\draw[black, fill=white] (2,1) circle (0.2);
\draw[black, fill=white] (3,0) circle (0.3);
\draw[black, fill=white] (3,0) circle (0.2);
\draw[black, fill=white] (1,2) circle (0.2);
\draw[black, fill=white] (2,2) circle (0.2);
\draw[black, fill=white] (3,2) circle (0.2);
\draw[black, fill=white] (1,3) circle (0.2);
\draw[black, fill=white] (4,1) circle (0.2);
\draw[black, fill=white] (4,2) circle (0.2);
\draw[black, fill=white] (3,3) circle (0.2);
\draw[black, fill=white] (4,3) circle (0.2);
\draw[black, fill=white] (0,2) circle (0.2);
\draw (2,1) node{$1$};
\draw (3,0) node{$2$};
\draw (1,2) node{$3$};
\draw (2,2) node{$4$};
\draw (3,2) node{$5$};
\draw (1,3) node{$6$};
\draw (4,1) node{$7$};
\draw (4,2) node{$8$};
\draw (3,3) node{$9$};
\draw (4,3) node{$10$};
\draw (0,2) node{$11$};
\end{tikzpicture}
\caption{\label{2-col} A 2-coloured rooted tree. }
\end{figure}

	\begin{prop} \label{3.4}
	The species of hollow hypertrees decorated by $\widehat{\prelie}$ is isomorphic to the species of 2-coloured rooted trees.
	\end{prop}
	
	\begin{proof}
	
A hollow hypertree with edges decorated by $\widehat{\prelie}$ is a hollow hypertree in which, for all edges $e$, the vertices of $e$ different from the gap or the petiole form a rooted tree.

Let us consider a hollow hypertree $H$ decorated by $\widehat{\prelie}$ on vertex set $V$. We call $E_0$, the set of edges between elements of $V$ in the rooted trees obtained from the decoration by $\widehat{\prelie}$. The graph $\left(V, E_0\right)$ is then a forest of trees obtained by deleting the edges of the hypertree $H$ and forgetting the roots. For any edge $e$ of $H$, we write $r_e$ for the root of the rooted tree in $e$ and $p_e$ for the petiole of $e$. Let $E_1$ be the set of edges between $r_e$ and $p_e$ for all edges $e$ of $H$. By definition of the sets, the intersection of $E_0$ with $E_1$ is empty. Moreover, to every path in $H$ corresponds a path in $\left(V, E_0 \cup E_1\right)$. As $H$ is a hypertree, the graph $\left(V, E_0 \cup E_1\right)$ is a tree $T$. We root that tree in the root $r$ of the tree in the hollow edge of $H$: $T$ is then a 2-coloured tree rooted in $r$. 

Conversely, let $T=\left(V, E_0 \cup E_1\right)$ be a 2-coloured rooted tree. The graph $\left(V, E_0\right)$ is a forest of trees: we can root these trees in their closest vertex from the root. Let us call $T_1, \ldots, T_n$ this forest, where $T_1$ is the tree rooted in the root of $T$. For all $i$ between $2$ and $n$, there is one vertex of $V$ closer from the root of $T$ than the root of $T_i$: we call this vertex $p_i$. Then, we consider the hypergraph whose set of vertices is $V$, with edges containing the vertices of $T_1$ or the vertices of a $T_i$ and $p_i$ for all $i$ between $2$ and $n$. Adding the edges of every $T_i$, we obtain a hypergraph decorated by $\widehat{\prelie}$. Moreover, paths in $T$ and in the hypergraph are the same: the hypergraph is then a hypertree.
	\end{proof}
	
	\begin{exple} The hollow tree with edges decorated by $\widehat{\prelie}$ associated to the 2-coloured rooted tree of figure \ref{2-col} is the hollow tree of figure \ref{htbij}. 
		\end{exple}
		
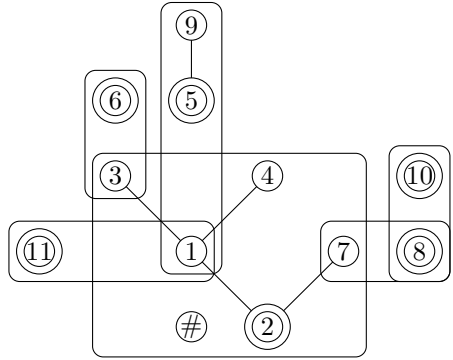
\begin{figure} 
\centering 
\begin{tikzpicture}[scale=1]
\draw[black, rounded corners] (0.7,-0.4) rectangle(4.3,2.3);
\draw[black, rounded corners] (-0.4,0.6) rectangle(2.3,1.4);
\draw[black, rounded corners] (1.6,0.7) rectangle(2.4,4.3);
\draw[black, rounded corners] (0.6,1.7) rectangle(1.4,3.4);
\draw[black, rounded corners] (3.7,0.6) rectangle(5.4,1.4);
\draw[black, rounded corners] (4.6,0.6) rectangle(5.4,2.4);
\draw (3,0) -- (2,1);
\draw (1,2) -- (2,1);
\draw (3,2) -- (2,1);
\draw (2,4) -- (2,3);
\draw (3,0) -- (4,1);
\draw[black, fill=white] (2,1) circle (0.2);
\draw[black, fill=white] (3,0) circle (0.3);
\draw[black, fill=white] (3,0) circle (0.2);
\draw[black, fill=white] (1,2) circle (0.2);
\draw[black, fill=white] (2,4) circle (0.2);
\draw[black, fill=white] (2,0) circle (0.2);
\draw[black, fill=white] (3,2) circle (0.2);
\draw[black, fill=white] (1,3) circle (0.3);
\draw[black, fill=white] (1,3) circle (0.2);
\draw[black, fill=white] (4,1) circle (0.2);
\draw[black, fill=white] (5,1) circle (0.3);
\draw[black, fill=white] (5,1) circle (0.2);
\draw[black, fill=white] (2,3) circle (0.3);
\draw[black, fill=white] (2,3) circle (0.2);
\draw[black, fill=white] (5,2) circle (0.3);
\draw[black, fill=white] (5,2) circle (0.2);
\draw[black, fill=white] (0,1) circle (0.3);
\draw[black, fill=white] (0,1) circle (0.2);
\draw (2,1) node{$1$};
\draw (2,0) node{$\#$};
\draw (3,0) node{$2$};
\draw (1,2) node{$3$};
\draw (2,3) node{$5$};
\draw (3,2) node{$4$};
\draw (1,3) node{$6$};
\draw (4,1) node{$7$};
\draw (5,1) node{$8$};
\draw (2,4) node{$9$};
\draw (5,2) node{$10$};
\draw (0,1) node{$11$};
\end{tikzpicture}
\caption{\label{htbij} A hollow hypertree decorated by $\widehat{\prelie}$.}
\end{figure}

Remark that this proposition gives the expression for $S^h_{\widehat{\prelie}}$ found previously in Proposition \ref{sprelie}.

Let us now link rooted edge-pointed hypertrees decorated by $\widehat{\prelie}$ with 2-coloured trees. We consider that edges of $E_0$ are blue and edges of $E_1$ are red.
\begin{prop} \label{2colpa}
The set of rooted edge-pointed hypertrees decorated by $\widehat{\prelie}$ is in one-to-one correspondence with the set of 2-coloured rooted trees whose root has exactly one blue child (and possibly some red children).
\end{prop}

\begin{proof}
Given a rooted edge-pointed hypertrees $H$ decorated by $\widehat{\prelie}$ with $k$ edges and $n$ vertices, we consider its pointed set of vertices $\left(r,V\right)$. We draw in blue all the edges obtained from the $\prelie$-structure in the edges of $H$: it gives a forest of $k+1$ trees with $n-k-1$ blue edges. We link $r$ with the root of the tree in the pointed edge of $H$ with a blue edge: it gives a forest of $k$ trees with $n-k$ blue edges. Now, for all non-pointed edges $e$, we link the petiole of $e$ with the root of the tree in $e$. We thus obtain a 2-coloured graph $G$ on $n$ vertices with $n-k + k-1$ edges. Let us show that this graph is connected. A path between the root and a vertex $x$ in $H$ is a sequence of paths in trees of the edges of $H$ from the root to a vertex which is the petiole of an edge and paths from the petiole of an edge to the root of the corresponding trees. As these paths also exist in $G$, we obtain a 2-coloured rooted trees whose root has exactly one blue child and possibly some red children.

Conversely, given such a 2-coloured rooted tree $T$, we delete red edges of $T$. We root all connected component in the vertices that was the closest in $T$ from the root of $T$. For every connected component $C_k$, we put $C_k$ and the parent of the root of $C_k$, if it exists, in the edge of a hypergraph $H$. We root the hypergraph in the root of $T$, point the edge containing the blue edge adjacent to the root and delete this blue edge. We thus obtain a rooted edge-pointed hypertrees decorated by $\widehat{\prelie}$. Indeed, the hypergraph is rooted and edge-pointed and every edge contains a rooted tree and a vertex so that the hypergraph can be seen as decorated by $\widehat{\prelie}$. Moreover, every vertex connected by an edge in $T$ are in the same edge in $H$ so $H$ is connected. Finally, a cycle in $H$ would come from a cycle in $T$, which does not exist so $H$ is a hypertree. The weight of the hypertree corresponds to the number of red edges.
\end{proof}

		\subsubsection{Computation of generating series of edge-pointed hypertrees}
		
	We now use the analogy with 2-coloured rooted trees to compute the value of the other generating series: 
		\begin{prop}		
		The generating series of the species of rooted edge-pointed hypertrees decorated by $\widehat{\prelie}$ is given by: 
		\begin{equation*}
S^{re}_{\widehat{\prelie}}=x+\sum_{n \geq 2} n\left(n+tn-1\right)^{n-3}\left(n-1\right)\left(1+2t\right) \frac{x^n}{n!}.
	\end{equation*}

		The generating series of the species of edge-pointed hypertrees decorated by $\widehat{\prelie}$ is given by: 
		\begin{equation*}
S^e_{\widehat{\prelie}}=x+\sum_{n \geq 2}\left(n+tn-1\right)^{n-3}\left(n-1\right)\left(1+tn\right) \frac{x^n}{n!}.
	\end{equation*}
		\end{prop}
		
		\begin{proof}
		\begin{itemize}
\item We linked rooted edge-pointed hypertrees decorated by $\widehat{\prelie}$ with 2-coloured rooted trees whose root has exactly one blue child (and possibly some red children) in Proposition \ref{2colpa}. We count the 2-coloured rooted trees using the number of red edges as a weight on 2-coloured trees.

Then, there is $n$ ways to choose the root of the 2-coloured rooted tree. If the root have $k$ children, $k-1$ of them are red, there is $k$ ways to choose the blue one and the set of the children forms a forest of $k$ 2-coloured rooted trees on $n-1$ vertices.

Let us count the forests of $k$ 2-coloured rooted trees on $n-1$ vertices. We use Formula \eqref{eprelie}: 
\begin{equation*}
E_{\prelie}\left(k,n-1\right)= \binom{n-2}{k-1} \left(n-1\right)^{n-1-k}.
\end{equation*}
There is $n-1-k$ edges in the forest and each of them is either blue or red.
Hence the number of forests of $k$ 2-coloured rooted trees on $n-1$ vertices is $\binom{n-2}{k-1} \times\left(n-1\right)^{n-1-k}\left(1+t\right)^{n-1-k}$.

Finally, the $n$-th coefficient of the series is: 
\begin{equation*}
\left(S^{re}_{\widehat{\prelie}}\right)_n = n \sum_{k=1}^{n-1} k t^{k-1} \binom{n-2}{k-1} \bigl(\left(n-1\right)\left(1+t\right) \bigr) ^{n-1-k}.
\end{equation*}

A quick computation using re-indexation and the binomial theorem gives the result.
		
\item We obtain the second generating series thanks to the dissymetry principle of Proposition \ref{principe de dissymétrie}.\qedhere
		\end{itemize}
		\end{proof}
	
		\subsubsection{Refinement}

We now use weights defined in Part \ref{refinement} on rooted and hollow hypertrees. With these weights, we will obtain the same results as in \cite{JMcCM}.
		
	\begin{prop}
	The generating series of the weighted species of hollow hypertrees decorated by $\widehat{\prelie}$ is given by: 
	\begin{equation*}
S^h_{\widehat{\prelie}}=x+\sum_{n \geq 2}\bigl(\left(x_1+\cdots + x_n\right)t+n\bigr)^{n-1} \frac{x^n}{n!}.
	\end{equation*}
	
		The generating series of the weighted species of rooted hypertrees decorated by $\widehat{\prelie}$ is given by: 	
		\begin{equation*}
S^r_{\widehat{\prelie}}=x+\sum_{n \geq 2}	\left(x_1+\cdots + x_n\right) \bigl(\left(x_1+\cdots + x_n\right)t+n-1\bigr)^{n-2} \frac{x^n}{n!}.
	\end{equation*}
	\end{prop}
	Hence the weighted rooted hypertrees enumerated in \cite[Theorem 3.9]{JMcCM} are in bijection with the $\widehat{\prelie}$-decorated rooted hypertrees.
	
	\begin{proof}
We use the formulas of Corollary \ref{formule Spx}. 
\end{proof}

We can separate the second $n$ into $y_1, \ldots,y_n$ by introducing the following weight on $\widehat{\prelie}$-decorated hollow hypertrees. First, we need some definitions: 

\begin{defi}
Given a rooted forest and a vertex $v$ of it, a \emph{child} $v'$ of the vertex $v$ is a vertex $v'$ linked to $v$ by an edge and such that $v$ is on the path from $v'$ to the root. The degree of $v$ in the rooted forest is the number of children of $v$ in it.
\end{defi}

Using this definition, we can define the following weight on the set of forests of rooted trees on a set of vertices $V$.

\begin{defi} \label{WF}
Let $F$ be a forest of rooted trees on a set of vertices $V$, we define the following weight on it: 
\begin{equation*}
W\left(F\right)=\prod_{i \in V} y_i^{s\left(i\right)},
\end{equation*}
where $s\left(i\right)$ is the number of children of the vertex $i$.
\end{defi}

The bijection between $\widehat{\prelie}$-decorated hollow hypertrees and box trees and set of decorated sets of Proposition \ref{bthhypertree} and the weights introduced in Definitions \ref{WBT} and \ref{WF} give a weight on the set of $\widehat{\prelie}$-decorated hollow hypertrees: 

\begin{defi}
Let $H$ be a $\widehat{\prelie}$-decorated hollow hypertree on a set of vertices $V$, we define the following weight on it: 
\begin{equation*}
W\left(T\right)=\prod_{i \in V} x_i^{p\left(i\right)} y_i^{s\left(i\right)},
\end{equation*}
where $p\left(i\right)$ is the number of edges of $H$ whose petiole is $i$ and $s\left(i\right)$ is the number of child of $v$ in the tree of the $\widehat{\prelie}$ decoration.
\end{defi}

Using the bijection with box trees and sets of decorated sets, we can compute the generating series of the weighted species of $\widehat{\prelie}$-decorated hollow hypertrees.

\begin{prop}
	The generating series of the weighted species of hollow hypertrees whose edges are decorated by $\widehat{\prelie}$ is given by: 
	\begin{equation*}
S^c_t=x+\sum_{n \geq 2}\bigl(\left(x_1+\cdots + x_n\right)t+y_1 + \cdots y_n\bigr)^{n-1} \frac{x^n}{n!}.
	\end{equation*}
\end{prop}

\begin{proof}
In Stanley's book \cite[proposition 5.3.2]{stan}, the formula for the number of forest of $k$ trees on $n$ vertices is given by: 
\begin{equation}
E_{\prelie}\left(t,k,n\right)=\binom{n-1}{k-1}\left(\sum_{j =1}^n \sum_{i \in V_j} y_i\right)^{n-k}.
\end{equation}
We use Corollary \ref{formule Spx} to obtain the expected results.
\end{proof}

		\subsubsection{Computation of cycle index series of hypertrees decorated by $\widehat{\prelie}$}

The reader may consult Appendix \ref{Annexe ind cycl} for basic definitions on cycle index series. We denote the cycle index series of usual species in the same way as the species itself. We compute the cycle index series of hypertrees decorated by $\widehat{\prelie}$. We do not write the argument of the cycle index series $\left(t, p_1, p_2, \ldots\right)$ in this subsection.
		
		\begin{prop}\label{3.13}
		
The cycle index series of hollow hypertrees decorated by $\widehat{\prelie}$ is given by: 
				\begin{equation}
	Z^h_{\widehat{\prelie}}= \frac { 1}{1+t}\prelie \circ\left(1+t\right) p_1.
\end{equation}
\end{prop}

\begin{proof} By Proposition \ref{3.4}, the cycle index series of hollow hypertrees decorated by $\widehat{\prelie}$ is given by the cycle index series of 2-coloured rooted trees.
\end{proof}

Let us define the following expressions, for $\lambda$ a partition of an integer $n$, written $\lambda \vdash n$: 
\begin{equation*}
f_k\left(\lambda\right)=\sum_{l | k} l \lambda_l
\end{equation*}
and
\begin{equation*}
P_k\left(\lambda\right)=\bigl((1+t^k)f_k\left(\lambda\right)-1 \bigr) ^{\lambda_k} -k \lambda_k(t^k+1) \times \bigl((1+t^k)f_k\left(\lambda\right)-1 \bigr) ^{\lambda_{k}-1}.
\end{equation*}

We obtain the following expression for the cycle index series of hypertrees decorated by $\widehat{\prelie}$: 

\begin{prop} The cycle index series of rooted hypertrees decorated by $\widehat{\prelie}$ is given by:
		\begin{equation}
		Z^r_{\widehat{\prelie}}= \sum_{n \geq 1} \sum_{\lambda \vdash n, \lambda_1 \neq 0} \lambda_1\left(\lambda_1 t +\lambda_1 -1\right)^{\lambda_1 - 2} \prod_{k \geq 2} P_k\left(\lambda\right) \frac{p_\lambda}{z_\lambda}.
		\end{equation}		

The cycle index series of edge-pointed rooted hypertrees decorated by $\widehat{\prelie}$ is given by:

		\begin{equation}
		Z^{re}_{\widehat{\prelie}}= \sum_{n \geq 1} \sum_{\lambda \vdash n, \lambda_1 \neq 0} \lambda_1\left(\lambda_1 -1\right)\left(2t+1\right)\left(\lambda_1 + \lambda_1 t-1\right)^{\lambda_1 -3}	 \prod_{k \geq 2} P_k\left(\lambda\right) \frac{p_\lambda}{z_\lambda}.
\end{equation}

The cycle index series of edge-pointed hypertrees decorated by $\widehat{\prelie}$ is given by:

		\begin{equation}
		Z^e_{\widehat{\prelie}}= \sum_{n \geq 1} \sum_{\lambda \vdash n, \lambda_1 \neq 0}\left(\lambda_1 -1\right)\left(1+\lambda_1 t\right)\left(\lambda_1 + \lambda_1 t-1\right)^{\lambda_1 -3}	 \prod_{k \geq 2} P_k\left(\lambda\right) \frac{p_\lambda}{z_\lambda}.
\end{equation}

The cycle index series of hypertrees decorated by $\widehat{\prelie}$ is given by:

		\begin{equation}
		Z_{\widehat{\prelie}}= \sum_{n \geq 1} \sum_{\lambda \vdash n, \lambda_1 \neq 0}\left(\lambda_1 t +\lambda_1 -1\right)^{\lambda_1 - 2} \prod_{k \geq 2} P_k\left(\lambda\right) \frac{p_\lambda}{z_\lambda}.
\end{equation}

		\end{prop}

\begin{proof}
We use Proposition \ref{décomp} and the cycle index series of hollow hypertrees decorated by $\widehat{\prelie}$ to compute the other series.

\begin{itemize}
\item The series $Z^r_{\widehat{\prelie}}$ satisfies:
\begin{equation*}
 Z^r_{\widehat{\prelie}} = \frac{p_1}{t} \times \biggl( \left(1+\comm\right) \circ\left(t \times Z^h_{\widehat{\prelie}}\right) \biggr).
\end{equation*}
Using the result of Proposition \ref{3.13}, it gives: 
\begin{equation*}
 Z^r_{\widehat{\prelie}} = \frac{p_1}{t} \times \biggl( \left(1+\comm\right) \circ\left(\frac{t}{1+t} \prelie \circ\bigl(\left(1+t\right)p_1\bigr)\right) \biggr).
\end{equation*}

We now use methods from the proof of the lemma 4 in the article of V. Dotsenko and A. Khoroshkin \cite{DotsKor} and from the proof of the proposition 7.2 in the article of F. Chapoton \cite{ChHyp}.

Consider $\lambda=\left(\lambda_1+1, \ldots, \lambda_r\right)$. We can suppose that, for $i \geq r$, we have $\lambda_i=0$. The coefficient of $t p_\lambda$ in $Z^r_{\widehat{\prelie}}$ is given by the multiple residue: 

\begin{equation*}
c_\lambda=\int\left(1+\comm\right) \circ \frac{t}{1+t} \prelie \circ\bigl(\left(1+t\right)p_1\bigr) \prod_{i=1}^{r} \frac{dp_i}{p_i^{\lambda_i +1}},
\end{equation*}
with $1+\comm = \prod_{k \geq 1} \exp\left({p_k}/{k}\right)$.

We use the substitution $y_l=p_l \circ t(1+t)^{-1} \prelie \circ\bigl(\left(1+t\right)p_1\bigr)$. By the Koszul duality of operads applied to the operad $\prelie$ in the article of F. Chapoton and M. Livernet \cite{ChLiv}, the cycle index series $\prelie$ satisfies $\bigl( -p_1 \left(1+\comm\right) \bigr) \circ -\prelie = p_1$. Therefore the substitution is given by for all $k \geq 1$ by: 
\begin{equation*}
p_k = \frac{y_k}{t^k} \exp\left(-\sum_{l\geq 1} \frac{1+t^{kl}}{t^{kl}} \frac{y_{kl}}{l}\right).
\end{equation*}
We thus obtain:
\begin{equation*}
c_\lambda=\int \prod_{k=1}^r \exp\left( \frac{y_k}{k} + \lambda_k \sum_{l \geq 1} \frac{1+t^{kl}}{t^{kl}} \frac{y_{kl}}{l}\right) \left(\frac{t^k}{y_k}\right)^{\lambda_k} \left(1-y_k \frac{1+t^k}{t^k}\right) y_k^{-1}\prod_{k=1}^{r} {dy_k}.
\end{equation*}
Then we can rewrite: 
\begin{equation*}
\sum_{k \geq 1 } k \lambda_k \sum_{l \geq 1} \frac{1+t^{kl}}{t^{kl}} \frac{y_{kl}}{kl} = \sum_{k \geq 1} \frac{1+t^{k}}{t^{k}} \left( f_k\left(\lambda\right) -1\right) \frac{y_{k}}{k},
\end{equation*}
with $f_k\left(\lambda\right)= \left( \sum_{l | k, l >1} l \lambda_l \right) + \lambda_1 +1$.
Therefore, $c_\lambda$ is given by the residue: 
\begin{equation*}
c_\lambda= \prod_{k=1}^{r} \int \exp \left(\frac{y_{k} a_k}{k} \right) t^{k \lambda_k}\left( 1 - y_k \frac{1+t^k}{t^k} \right)\frac{dy_k}{y_k^{\lambda_k +1}},
\end{equation*}
with $a_k=1+  \left(1+t^{k}\right) \left(t^k \right)^{-1}\left( f_k\left(\lambda\right) -1\right)$. This gives the expected result as the residue of $\exp\left(az\right)z^{-n}$, for a constant $a$, is given by ${a^{n-1}}/{\left(n-1\right)!}$.

\item The relation \eqref{relpa} gives the following relation for the series $Z^{re}_{\widehat{\prelie}}$: 
\begin{equation*}
Z^{re}_{\widehat{\prelie}}= p_1 \times \left( \bigl( p_1\left(1+ \comm\right)\bigr) \circ \frac{t}{1+t} \prelie \circ\bigl(\left(1+t\right)p_1\bigr) \right).
\end{equation*}

We use the same method as for rooted hypertrees with the same substitution to obtain the result.

\item The relation \eqref{rela} gives the following relation for the series $Z^{e}_{\widehat{\prelie}}$: 
\begin{equation*}
Z^{e}_{\widehat{\prelie}}=\widehat{\prelie} \circ t Z^r_{\widehat{\prelie}}.
\end{equation*}

However the relation (50) in the article of F. Chapoton \cite{ChAOp} gives: 
\begin{equation*}
\widehat{\prelie} = p_1 \left( 1+ \prelie + \frac{1}{\prelie}\right).
\end{equation*}

Hence the series $Z^{e}_{\widehat{\prelie}}$ satisfies: 

\begin{equation*}
Z^{e}_{\widehat{\prelie}}= t Z^r_{\widehat{\prelie}} \left( 1+ \prelie \circ t Z^r_{\widehat{\prelie}} + \frac{1}{\prelie \circ t Z^r_{\widehat{\prelie}}}\right).
\end{equation*}

Moreover, using the expression of $Z^h_{\widehat{\prelie}}$ in terms of $Z^r_{\widehat{\prelie}}$ of relation \eqref{relhc}, and the one of the proposition, we obtain: 
\begin{equation*}
Z^{e}_{\widehat{\prelie}}= p_1 \times \left(\left(1+\comm\right)\left( 1+\frac{p_1}{t} + \frac{t}{p_1}\right)\right) \circ \frac{t}{1+t} \prelie \circ\bigl(\left(1+t\right)p_1\bigr).
\end{equation*}

We use the same method as for rooted hypertrees with the same substitution to obtain the result.

\item The series $Z_{\widehat{\prelie}}$ is obtained by using the dissymetry principle \ref{principe de dissymétrie}. \qedhere
\end{itemize}
\end{proof}
	
	\subsection{$\widehat{\lie}$-decorated hypertrees}
	
	 \label{Lie}

In this part, we study $\widehat{\lie}$-decorated hypertrees, where $\lie$ is the species associated to the following cycle index series: 
\begin{equation*}
\lie= \sum_{k \geq 1} \frac{\mu\left(k\right)}{k} \log\left(1-p_k\right),
\end{equation*}
where $\mu$ is the Möbius function, whose value is $(-1)^p$ on square-free positive integers with $p$ prime factors, and $0$ otherwise.
		\subsubsection{Computation of generating series of hypertrees decorated by $\widehat{\lie}$}
	Applying the results of the second section, we obtain the following proposition: 
	
	\begin{prop}
	The generating series of the species of rooted hypertrees decorated by $\widehat{\lie}$ is given by: 
	\begin{equation*}
S^r_{\widehat{\lie}}= \sum_{n \geq 1}	\frac{1}{t} \prod_{k=0}^{n-2}\left(nt+k\right) \times \frac{x^n}{n!}.
	\end{equation*}
	
	The generating series of the species of hypertrees decorated by $\widehat{\lie}$ is given by: 
	\begin{equation*}
S_{\widehat{\lie}}=\sum_{n \geq 1}	\prod_{k=1}^{n-2}\left(nt+k\right) \times \frac{x^n}{n!}.
	\end{equation*}
	
		The generating series of the species of hollow hypertrees decorated by $\widehat{\lie}$ is given by: 
	\begin{equation*}
S^h_{\widehat{\lie}}=\sum_{n \geq 1} \prod_{k=1}^{n-1}\left(nt+k\right) \frac{x^n}{n!}.
	\end{equation*}

		The generating series of the species of rooted edge-pointed hypertrees decorated by $\widehat{\lie}$ is given by: 
	\begin{equation*}
S^{re}_{\widehat{\lie}}=\sum_{n \geq 2} \sum_{p=1}^{n-1} \binom{n}{p} \prod_{k=0}^{n-2} \prod_{l=1}^{n-p-1}\left(pt+k\right)\bigl( \left(n-p\right)t+l \bigr) \frac{x^n}{n!}.
	\end{equation*}

		The generating series of the species of edge-pointed hypertrees decorated by $\widehat{\lie}$ is given by: 
	\begin{equation*}
S^{e}_{\widehat{\lie}}=\sum_{n \geq 2} \prod_{k=0}^{n-2}\left(nt+k\right) - \sum_{p=0}^{n-1} \binom{n}{p} \prod_{k=0}^{n-2} \prod_{l=1}^{n-p-1}\left(pt+k\right)\bigl(\left(n-p\right)t+l\bigr)\frac{x^n}{n!}.
	\end{equation*}
	
	\end{prop}	
		
		\begin{proof}
		The three first results are obtained by applying the expressions of Formula \ref{formule Sp} and the formula \eqref{formule S} of Theorem \ref{cor S} with $E_{\widehat{\lie}}\left(k,n\right)=\left(-1\right)^{k+n} s\left(k,n\right)$, where $s\left(k,n\right)$ is a Stirling number of the first kind. Indeed, the Stirling numbers of the first kind satisfy the following classical equation: 
		\begin{equation*}
		\sum_{k=1}^{n} s\left(k,n\right) x^k = \prod_{k=0}^{n-1}\left(x-k\right).
		\end{equation*}
		
		The fourth one is obtained using the relation \eqref{relpa} and the last one is obtained by using the dissymetry principle of Proposition \ref{principe de dissymétrie}.
\end{proof}			
	
		\subsubsection{Computation of cycle index series of hypertrees decorated by $\widehat{\lie}$}
		
Let us now study the action of the symmetric group on these hypertrees.		
		
Let us define the following expressions, for $\lambda=(\lambda_1, \lambda_2, \ldots)$ a partition of an integer $n$: 
\begin{equation*}
f_k\left(\lambda\right)=\sum_{l | k} l \lambda_l
\end{equation*}
and
\begin{equation*}
\varphi_i\left(\lambda\right) = \sum_{k | i} \frac{t^k}{i} \mu\left(\frac{i}{k}\right) f_k\left(\lambda\right),
\end{equation*}

where $\mu$ is the Möbius function.

Using these expressions, we have the following relations: 

		\begin{prop}
	The cycle index series of hollow hypertrees decorated by $\widehat{\lie}$ is given by: 

\begin{multline*}
	Z^h_{\widehat{\lie}}= \sum_{n \geq 1} \sum_{\lambda \vdash n } \sum_{p\geq 1} \frac{\mu\left(p\right)}{p} \prod_{i=1, i \neq p}^{r} \left( \binom{\varphi_i\left(\lambda\right) + \lambda_i -1}{\lambda_i} - t^i \binom{\varphi_i\left(\lambda\right) + \lambda_i -1}{\lambda_i -1}\right) \\ \times \sum_{q=1}^{\lambda_p}\frac{1}{q} \left( \binom{\varphi_p\left(\lambda\right) + \lambda_p -q -1}{\lambda_p-q} - t^q \binom{\varphi_p\left(\lambda\right) + \lambda_p -q -1}{\lambda_p -q-1}\right) \frac{p_\lambda}{z_\lambda}.
\end{multline*}

The cycle index series of rooted hypertrees decorated by $\widehat{\lie}$ is given by: 
		\begin{equation*}
		Z^r_{\widehat{\lie}}= \sum_{n \geq 1} \sum_{\lambda \vdash n} \frac{1}{t} \left( \prod_{k=0}^{\lambda_1-2}\left(\lambda_1 t +k\right) \right) \prod_{i \geq 2} \left( \binom{\varphi_i\left(\lambda\right) + \lambda_i -1}{\lambda_i} - t^i \binom{\varphi_i\left(\lambda\right) + \lambda_i -1}{\lambda_i-1} \right) \frac{p_\lambda}{z_\lambda}.
		\end{equation*}		
		
		\end{prop}
		
		\begin{proof}
\begin{itemize}
\item The equations of Corollary \ref{coreq} give: 
\begin{equation*}
t Z^r_{\widehat{\lie}} = p_1 +p_1 \times \comm \circ \left( t \times \lie \circ t Z^r_{\widehat{\lie}} \right)
\end{equation*}
and
\begin{equation*}
Z^h_{\widehat{\lie}} = \lie \circ t Z^r_{\widehat{\lie}}.
\end{equation*}

Consider $\lambda=\left(\lambda_1, \ldots, \lambda_r\right)$. We can suppose that, for $i \geq r$, we have $p_i=0$. The coefficient of $p_{\lambda}$ in $Z^r_{\widehat{\lie}}$ is given by the residue $c_\lambda$: 
\begin{equation*}
c_\lambda = \int Z^r_{\widehat{\lie}} \prod_{i=1}^{r} \frac{dp_i}{p_i^{\lambda_i +1}}.
\end{equation*}

We use the substitution $y_i=p_i \circ t Z^r_{\widehat{\lie}}$. Let us compute first $\left(1+ \comm\right) \circ t \lie$.
\begin{center}
\begin{align*}
\left(1+ \comm\right) \circ t \lie & = \exp \left( \sum_{k \geq 1} \frac{t^k}{k} p_k \circ \left( - \sum_{l \geq 1} \frac{\mu\left(l\right)}{l} \log\left(1-p_l\right) \right) \right) \\
& = \exp \left( - \sum_{k \geq 1} \sum_{l \geq 1} \frac{t^k}{kl} \mu\left(l\right) \log\left(1-p_{kl}\right) \right) \\
& =\prod_{k,l \geq 1}\left(1-p_{kl}\right)^{- \frac{t^k \mu\left(l\right)}{kl}}.
\end{align*}
\end{center}

Hence the substitution is given by: 
\begin{equation*}
p_i = y_i \prod_{k,l \geq 1}\left(1-y_{kli}\right)^{ \frac{t^{ki} \mu\left(l\right)}{kl}}.
\end{equation*}

With this substitution, we obtain: 
\begin{align*}
c_\lambda & = \int \frac{y_1}{t}\prod_{i=1}^{r } \prod_{k,l \geq 1}\left(1-y_{kli}\right)^{ - \frac{\lambda_i t^{ki} \mu\left(l\right)}{kl}} \left( 1- \frac{t^i y_i}{1-y_i} \right) \frac{dy_i}{y_i^{\lambda_i +1}}.
\end{align*}

We then separate the terms in each of the $y_j$. It gives: 

\begin{multline*}
 c_\lambda = \int \frac{y_1}{t} \left(1-y_1\right)^{- \lambda_1 t} \left( 1- \frac{t y_1}{1-y_1}\right) \frac{d y_1}{y_1 ^{\lambda_1 +1}} \\ \times \prod_{i=2}^{r } \int\left(1-y_{i}\right)^{ - \varphi_i\left(\lambda\right)} \left( 1- \frac{t^i y_i}{1-y_i}\right) \frac{dy_i}{y_i^{\lambda_i +1}},
\end{multline*}

where we set:
\begin{equation*}
\varphi_i\left(\lambda\right) = \sum_{j|k|i} \frac{t^k}{i} \mu \left( \frac{i }{k}\right) j \lambda_j =\sum_{k | i} \frac{t^k}{i} \mu\left(\frac{i}{k}\right) f_k\left(\lambda\right).
\end{equation*}

We now compute the following integral: 

\begin{align*}
\int y_p^q & \left(1-y_p\right) ^{-\varphi_p\left(\lambda\right) -1}\left(1-y_p\left(1+t^p\right)\right) \frac{dy_p}{y_p^{\lambda_p+1}} \\ & = \sum_{j \geq 0} \binom{\varphi_p\left(\lambda\right) + j}{j} \left( \int y_p^{q+j- \lambda_p -1} dy_p -\left(1+t^p\right) \int y_p^{q+j-\lambda_p} dy_p\right)
\\ & = \binom{\varphi_p\left(\lambda\right) + \lambda_p -q-1}{\lambda_p -q} - t^p \binom{\varphi_p\left(\lambda\right) + \lambda_p -q-1}{\lambda_p -q-1}.
\end{align*}

This computation applied to $c_\lambda$ give the expected result for $Z^r_{\widehat{\lie}}$.

 \item Calling $\lambda=\left(\lambda_1, \ldots, \lambda_r\right)$. We can suppose that, for $i \geq r$, we have $p_i=0$. The coefficient of $ p_{\lambda}$ in $Z^h_{\widehat{\lie}}$ is given by the residue $d_\lambda$: 
\begin{equation*}
d_\lambda=\int \lie \circ t Z^r_{\widehat{\lie}} \prod_{i=1}^{r} \frac{dp_i}{p_i^{\lambda_i +1}}.
\end{equation*}

To compute $d_\lambda$, we use the same substitution $y_i=p_i \circ t Z^r_{\widehat{\lie}}$. It gives, expanding $\log$: 
\begin{align*}
d_\lambda & = -\int \sum_{p \geq 1} \frac{\mu\left(p\right)}{p} \log\left(1-y_p\right) \prod_{i=1}^{r} \prod_{k,l \geq 1}\left(1-y_{kli}\right)^{- \frac{i \lambda_i t^{ki} \mu\left(l\right)}{kli}} \left( 1- \frac{t^i y_i}{1-y_i}\right) \frac{dy_i}{y_i^{\lambda_i +1}} \\
& = \sum_{p \geq 1} \frac{\mu\left(p\right)}{p} \sum_{q \geq 1} \frac{1}{q} \int y_p^q \prod_{i=1}^{r}\left(1-y_i\right)^{- \varphi_i(\lambda)} \left( 1- \frac{t^i y_i}{1-y_i}\right) \frac{dy_i}{y_i^{\lambda_i +1}}.
\end{align*}

This gives the expected result.\qedhere
\end{itemize}
\end{proof}	

		\subsubsection{Link with the operad $\Lambda$}
		
In this part, we link hypertrees decorated by $\widehat{\Sigma \lie}$ and $\widehat{\lie}$ with an operad $\Lambda$, defined in the article of F. Chapoton \cite{OpDiff}. We recommend the reader to consult the reminder on cycle index series in Appendix \ref{Annexe ind cycl} for the definition of the suspension $\Sigma$.

\begin{prop} The hollow hypertrees decorated by $\widehat{\Sigma \lie}$ and $\widehat{\lie}$ are related by the following relation: 
\begin{equation*}
Z^h_{\widehat{\lie}}\left(t, p_1, p_2, \ldots\right) = \Sigma Z^h_{\widehat{\Sigma \lie}}\left(-t, p_1, p_2, \ldots\right).
\end{equation*}
\end{prop}

\begin{proof}
According to Corollary \ref{coreq}, the series $Z^h_{\widehat{\Sigma \lie}}$ satisfies: 
\begin{equation*}
Z^h_{\widehat{\Sigma \lie}}\left(t, p_1, p_2, \ldots\right) = \Sigma \lie \circ \left(p_1 \times \left(1+ \comm\right) \circ\left(t Z^h_{\widehat{\Sigma \lie}}\left(t, p_1, p_2, \ldots\right) \right) \right).
\end{equation*}

Hence, applying the suspension and using Proposition \ref{prop suspension}, we obtain: 
\begin{align*}
\Sigma Z^h_{\widehat{\Sigma \lie}}\left(t, p_1, p_2, \ldots\right)& = \lie \circ \left(- p_1 \times \Sigma\left(1+ \comm\right) \circ\left(t \Sigma Z^h_{\widehat{\Sigma \lie}}\left(t, p_1, p_2, \ldots\right)\right) \right), \\
& = \lie \circ \left(p_1 \times \left(1+ \comm\right) \circ\left(-t \Sigma Z^h_{\widehat{\Sigma \lie}}\left(t, p_1, p_2, \ldots\right)\right) \right).
\end{align*}

The last equation is the equation defining $Z^h_{\widehat{\lie}}\left(-t, p_1, p_2, \ldots\right)$, so we get the expected result.
\end{proof}

Let us now introduce the operads $\pasc$ and $\Lambda$, which have been defined in \cite{OpDiff}. These two operads are symmetric operads in the category of graded vector spaces defined by some binary generators and quadratic relations \footnote{These presentations can be found in \cite{OpDiff}}. The operad $\pasc$ admits an explicit basis indexed by non-empty subsets, and its underlying species is given by : 
\begin{equation*}
\pasc = (1+\comm) \Sigma_t \comm.
\end{equation*}

The operad $\Lambda$ is defined as the quadratic dual of the operad $\pasc$. In an unpublished paper \cite{Dotp}, V. Dotsenko has proved that these operads are Koszul. This result gives the following proposition.

\begin{prop}
Hollow hypertrees decorated by $\widehat{\Sigma \lie}$ are related to the operad $\Lambda$ by: 
\begin{align*}
Z^h_{\widehat{\Sigma \lie}}\left(t, p_1, p_2, \ldots, p_i, \ldots \right) & = t^{-1} \Lambda\left(t^{-1},t p_1, \ldots, t^i p_i, \ldots\right) \\ & = \Sigma \Sigma_t \Lambda\left(t^{-1}, p_1, \ldots\right).
\end{align*}
\end{prop}

\begin{proof} Corollary \ref{coreq} give the following relation, as $\comm \circ \Sigma \lie = p_1$: 
\begin{align*}
\left( Z^h_{\widehat{\Sigma \lie}} \right)^{-1} & = \frac{\comm}{1+ \comm \circ t p_1}, \\
 & = \left( \prod_{k \geq 1} \exp \left( {p_k}/{k} \right) -1 \right) \times \prod_{k \geq 1} \exp \left(- {t^k p_k}/{k} \right).
\end{align*}
Now, as $\Lambda$ is Koszul, it satisfies $\Sigma \pasc \circ \Lambda = p_1$, which gives the relation: 
\begin{align*}
t p_1 & = \prod_{k \geq 1} \exp \left(- \frac{p_k}{k} \circ \Lambda\left(t,p_1, \ldots, p_i, \ldots\right) \right) \\ & \times \prod_{k \geq 1} \exp \left( \frac{t^k}{k} p_k \circ \Lambda\left(t,p_1, \ldots, p_i, \ldots\right) \right) -1.
\end{align*}
Hence, substituting $t$ by $t^{-1}$ in this expression, as $1= t \times t^{-1}$, we obtain: 
\begin{align*}
\frac{p_1}{t} & = \prod_{k \geq 1} \exp \left(- \frac{t^k}{k} p_k \circ t^{-1}\Lambda\left(t^{-1},p_1, \ldots, \right) \right) \\ & \left(\prod_{k \geq 1} \exp \left( \frac{ p_k}{ k} \circ t^{-1} \Lambda\left(t^{-1},p_1, \ldots, \right) \right) -1\right) \\
& =\left( Z^h_{\widehat{\Sigma \lie}} \right)^{-1} \circ t^{-1} \Lambda\left(t^{-1},p_1, \ldots, p_i, \ldots\right).
\end{align*}

Composing by $t p_1$, we obtain that $\Sigma \Sigma_t \Lambda\left(t^{-1}, p_1, \ldots\right)$ is the inverse of $\left( Z^h_{\widehat{\Sigma \lie}} \right)^{-1}$ for the plethystic substitution.
\end{proof}

\section{Generalizations: Bi-decorated hypertrees}

	\subsection{Definitions of bi-decorated hypertrees and rooted or edge-pointed variants of them}

In this part, we generalize the decoration of edges studied previously to the decoration of edges and neighbourhood of vertices.

\begin{defi}
Given two species or linear species $\mathcal{S}_e$ and $\mathcal{S}_v$, a \emph{bi-decorated (rooted) hypertree} is obtained from a $\mathcal{S}_e$-edge-decorated (rooted) hypertree by choosing for every vertex $v$ of $H$ an element of $\mathcal{S}_v\left(E_v\right)$, where $E_v$ is the set of edges containing $v$.

A \emph{skew-bi-decorated rooted hypertree} is obtained from a $\mathcal{S}_e$-edge-decorated rooted hypertree $H$ by choosing for every vertex $v$ of $H$ an element of $\mathcal{S}'_v\left(E_v\right)$, where $E_v$ is the set of edges containing $v$.
\end{defi}

Let us remark that when $\mathcal{S}_v$ is the species $\comm$, bi-decorated rooted hypertrees and skew-bi-decorated rooted hypertrees are the same type of hypertrees: edge-decorated rooted hypertrees.

The map which associates to a finite set $I$ the set of bi-decorated (resp. bi-decorated rooted) hypertrees on $I$ is a species, called the $\left(\mathcal{S}_e, \mathcal{S}_v\right)$-bi-decorated (resp. bi-decorated rooted) hypertrees species and denoted by $\hsb$ (resp. $\hsbp$). The map which associates to a finite set $I$ the set of skew-bi-decorated rooted hypertrees on $I$ is a species, called the $\left(\mathcal{S}_e, \mathcal{S}_v\right)$-skew-bi-decorated rooted hypertrees species and denoted by $\hsbtp$.

\begin{exple} On the right side of figure \ref{bdex} is represented an example of bi-decorated hypertree with the neighbourhood of vertices decorated by the species of non-empty lists $\assoc$ and edges decorated by the species of non-empty pointed sets $\perm$. The pointed vertex of each edge is represented with a star $*$ next to it. The order of edges around a vertex is given by the numbers around the vertex, next to the edges. For instance, around the vertex $9$, the edge $\{8,9\}$ is the first in the list, $\{10,9\}$ is the second one and $\{11,9\}$ is the third and last one. 

On the left side of figure \ref{bdex}, we have represented an example of bi-decorated rooted hypertree, with edges decorated by the species of cycle and the neighbourhood of vertices decorated by the species of non-empty pointed sets $\perm$. When the vertex is in only one edge, this edge is the pointed edge, otherwise, we put a star $*$ near the pointed edge of each vertex. For instance, in the neighbourhood of the vertex $15$, the pointed edge is the edge $\{12,15\}$, which is also pointed in the neighbourhood of the vertex $12$.
\end{exple}

\begin{figure} 
\centering 
\begin{tikzpicture}[scale=1]
\draw [black, fill=gray!40] (6,2) -- (6,3) -- (7,3) -- (7,2) -- (6,2);
\draw[black, fill=gray!40] (7,0) -- (8,0) -- (7,1) -- (7,0);
\draw[black, fill=gray!40] (7,1) -- (7,2) -- (8,2) -- (7,1);
\draw (6,0) -- (6,1);
\draw (6,2) -- (6,1);
\draw (8,2) -- (9,2);
\draw (8,2) -- (8,1);
\draw (5,0) -- (6,1);
\draw (6,2) -- (6,1);
\draw (6,0) -- (6,1);
\draw[black, fill=white] (5,0) circle (0.2);
\draw[black, fill=white] (6,0) circle (0.2);
\draw[black, fill=white] (6,1) circle (0.2);
\draw[black, fill=white] (6,2) circle (0.2);
\draw[black, fill=white] (6,3) circle (0.2);
\draw[black, fill=white] (7,0) circle (0.2);
\draw[black, fill=white] (7,1) circle (0.2);
\draw[black, fill=white] (7,2) circle (0.2);
\draw[black, fill=white] (7,3) circle (0.2);
\draw[black, fill=white] (8,0) circle (0.2);
\draw[black, fill=white] (8,1) circle (0.2);
\draw[black, fill=white] (8,2) circle (0.2);
\draw[black, fill=white] (9,2) circle (0.2);
\draw (5,0) node{$11$};
\draw (5.1,0.3) node{$_1$};
\draw (6,1) node{$9$};
\draw (5.7,0.9) node{$*$};
\draw (6.1,0.3) node{$_1$};
\draw (5.9,0.3) node{$*$};
\draw (6.1,1.3) node{$*$};
\draw (5.9,0.7) node{$_3$};
\draw (6.1,0.7) node{$_2$};
\draw (5.9,1.3) node{$_1$};
\draw (6,0) node{$10$};
\draw (6,2) node{$8$};
\draw (6.3,2.2) node{$_1$};
\draw (6.1,1.7) node{$_2$};
\draw (6,3) node{$12$};
\draw (6.3,2.8) node{$_1$};
\draw (6.1,2.7) node{$*$};
\draw (7,0) node{$4$};
\draw (7.15,0.3) node{$_1$};
\draw (7.3,0.15) node{$*$};
\draw (7,1) node{$3$};
\draw (7.3,1.45) node{$_2$};
\draw (7.15,1.25) node{$*$};
\draw (7.15,0.6) node{$_1$};
\draw (7,2) node{$1$};
\draw (6.7,2.2) node{$_2$};
\draw (7.3,1.85) node{$_1$};
\draw (7,3) node{$13$};
\draw (6.7,2.8) node{$_1$};
\draw (8,0) node{$5$};
\draw (7.7,0.15) node{$_1$};
\draw (8,1) node{$6$};
\draw (8.1,1.3) node{$*$};
\draw (7.9,1.3) node{$_1$};
\draw (8,2) node{$2$};
\draw (7.7,1.85) node{$_1$};
\draw (8.3,2.1) node{$_2$};
\draw (8.1,1.7) node{$_3$};
\draw (9,2) node{$7$};
\draw (8.7,2.1) node{$*$};
\draw (8.7,1.9) node{$_1$};

\draw[black, fill=gray!40] (2,0) -- (-2,1) -- (-1,1) -- (2,0);
\draw[black, fill=gray!40] (2,0) -- (0,1) -- (1,1) -- (2,1) -- (2,0);
\draw[black, fill=gray!40] (2,0) -- (3,1) -- (4,1) -- (2,0);
\draw[black, fill=gray!40] (1,2) -- (0,3) -- (1,3) -- (2,3) -- (3,3) -- (1,2);
\draw (-1,1) -- (-1,2);
\draw (2,1) -- (1,2);
\draw (2,1) -- (3,2);

\draw[black, fill=white] (2,0) circle (0.3);
\draw[black, fill=white] (2,0) circle (0.2);
\draw[black, fill=white] (-2,1) circle (0.2);
\draw[black, fill=white] (-1,1) circle (0.2);
\draw[black, fill=white] (-1,2) circle (0.2);
\draw[black, fill=white] (0,1) circle (0.2);
\draw[black, fill=white] (1,1) circle (0.2);
\draw[black, fill=white] (3,1) circle (0.2);
\draw[black, fill=white] (2,1) circle (0.2);
\draw[black, fill=white] (4,1) circle (0.2);
\draw[black, fill=white] (1,2) circle (0.2);
\draw[black, fill=white] (3,2) circle (0.2);

\draw[black, fill=white] (0,3) circle (0.2);
\draw[black, fill=white] (1,3) circle (0.2);
\draw[black, fill=white] (2,3) circle (0.2);
\draw[black, fill=white] (3,3) circle (0.2);

\draw (2,0) node{$3$};
\draw (1.75,0.3) node{$*$};
\draw (-2,1) node{$10$};
\draw (-1,1) node{$9$};
\draw (-1.1,1.3) node{$*$};
\draw (0,1) node{$4$};
\draw (3,1) node{$5$};
\draw (1,1) node{$6$};
\draw (2,1) node{$15$};
\draw (1.8,1.3) node{$*$};
\draw (4,1) node{$1$};
\draw (-1,2) node{$11$};
\draw (3,2) node{$8$};
\draw (1,2) node{$12$};
\draw (1.2,1.7) node{$*$};

\draw (0,3) node{$14$};
\draw (1,3) node{$7$};
\draw (2,3) node{$13$};
\draw (3,3) node{$2$};
\end{tikzpicture}
\caption{\label{bdex} Bi-decorated hypertrees.}
\end{figure}
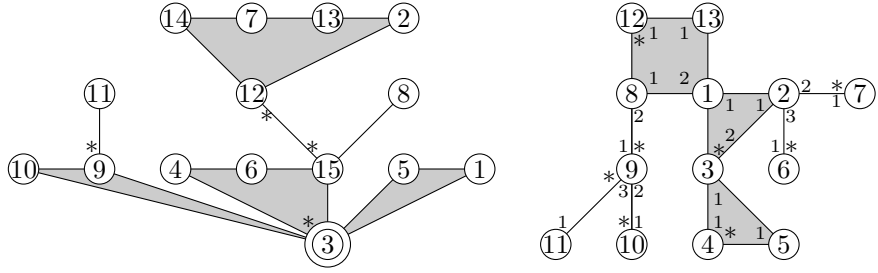

We now give definitions for edge-pointed bi-decorated hypertrees.

\begin{defi}
Given two species $\mathcal{S}_e$ and $\mathcal{S}_v$, a \emph{bi-decorated edge-pointed (rooted) hypertree} (resp. \emph{bi-decorated hollow hypertree}) is obtained from a $\mathcal{S}_e$-edge-decorated edge-pointed (rooted) hypertree (resp. a $\mathcal{S}_e$-edge-decorated hollow hypertree) $H$ by choosing for every vertex $v$ of $H$ an element of $\mathcal{S}_v\left(E_v\right)$, where $E_v$ is the set of edges containing $v$.

We can give an alternative definition. Given a vertex $v$, there is one edge containing $v$ which is the nearest from the pointed or hollow edge of $H$: let us call it the \emph{branch} of $v$. Then, a \emph{bi-decorated edge-pointed (rooted) hypertree} (resp. \emph{bi-decorated hollow hypertree}) is obtained from the hypertree $H$ by choosing for every vertex $v$ of $H$ an element of $\mathcal{S}'_v\left(E^l_v\right)$, where $E^l_v$ is the set of edges containing $v$ different from the branch.
\end{defi}

The map which associates to a finite set $I$ the set of bi-decorated edge-pointed (resp. edge-pointed rooted, resp. hollow) hypertrees on $I$ is a species, called the $\left(\mathcal{S}_e, \mathcal{S}_v\right)$-edge-decorated edge-pointed (resp. edge-pointed rooted, resp. hollow) hypertrees species and denoted by $\hsba$ (resp. $\hsbpa$, resp. $\hsbc$).

\begin{exple} We decorate the edge-pointed rooted hypertree on the left side of the figure \ref{bdrex} by the species of cycles on edges and around vertices. This give the planar hypertree on the right side of the figure \ref{bdrex} where the vertices and the edges are ordered in lists according to their place in the plane. The leftmost edge attached to the root is the pointed edge. For every vertex different from the root, the edge coming from the root to the vertex is the branch of the vertex and the other edges containing the vertex form a list.
\end{exple}

\begin{figure} 
\centering 
\begin{tikzpicture}[scale=1]
\draw[black, fill=gray!40] (0,1) -- (0,0) -- (1,0) -- (0,1);
\draw[black, fill=gray!40] (1,0) -- (1,1) -- (2,1) -- (2,0) -- (1,0);
\draw[black, fill=gray!40] (2,1) -- (2,2) -- (3,2) -- (2,1);
\draw[black, fill=gray!40] (2.1,1.2) -- (2.1,1.9) -- (2.8,1.9) -- (2.1,1.2);
\draw[black, fill=gray!40] (3,1) -- (4,1) -- (4,0) -- (3,1);
\draw (0,2) -- (0,1);
\draw (1,2) -- (2,2);
\draw (2,1) -- (3,1);
\draw (3,0) -- (3,1);
\draw (3,2) -- (4,2);

\draw[black, fill=white] (2,1) circle (0.3);
\draw[black, fill=white] (0,0) circle (0.2);
\draw[black, fill=white] (1,0) circle (0.2);
\draw[black, fill=white] (2,0) circle (0.2);
\draw[black, fill=white] (3,0) circle (0.2);
\draw[black, fill=white] (4,0) circle (0.2);
\draw[black, fill=white] (0,1) circle (0.2);
\draw[black, fill=white] (1,1) circle (0.2);
\draw[black, fill=white] (2,1) circle (0.2);
\draw[black, fill=white] (3,1) circle (0.2);
\draw[black, fill=white] (4,1) circle (0.2);
\draw[black, fill=white] (0,2) circle (0.2);
\draw[black, fill=white] (1,2) circle (0.2);
\draw[black, fill=white] (2,2) circle (0.2);
\draw[black, fill=white] (3,2) circle (0.2);
\draw[black, fill=white] (4,2) circle (0.2);

\draw (0,2) node{$12$};
\draw (0,1) node{$10$};
\draw (0,0) node{$11$};
\draw (1,2) node{$8$};
\draw (1,1) node{$5$};
\draw (1,0) node{$6$};
\draw (2,2) node{$2$};
\draw (2,1) node{$1$};
\draw (2,0) node{$7$};
\draw (3,2) node{$3$};
\draw (3,1) node{$4$};
\draw (3,0) node{$15$};
\draw (4,2) node{$9$};
\draw (4,1) node{$13$};
\draw (4,0) node{$14$};

\draw[black, fill=gray!40] (7,0) -- (5,1) -- (6,1) -- (7,0);
\draw[black, fill=gray!40] (7,0) -- (9,1) -- (11,1) -- (7,0);
\draw[black, fill=gray!40] (7,1) -- (7,2) -- (8,2) -- (7,1);
\draw[black, fill=gray!40] (10,1) -- (10,2) -- (11,2) -- (10,1);
\draw (7,0) -- (7,1);
\draw (5,1) -- (5,2);
\draw (6,1) -- (6,2);
\draw (7,1) -- (9,2);
\draw (10,2) -- (10,3);

\draw[black, fill=white] (7,0) circle (0.3);
\draw[black, fill=white] (7,0) circle (0.2);
\draw[black, fill=white] (5,1) circle (0.2);
\draw[black, fill=white] (6,1) circle (0.2);
\draw[black, fill=white] (7,1) circle (0.2);
\draw[black, fill=white] (9,1) circle (0.2);
\draw[black, fill=white] (10,1) circle (0.2);
\draw[black, fill=white] (11,1) circle (0.2);
\draw[black, fill=white] (5,2) circle (0.2);
\draw[black, fill=white] (6,2) circle (0.2);
\draw[black, fill=white] (7,2) circle (0.2);
\draw[black, fill=white] (8,2) circle (0.2);
\draw[black, fill=white] (9,2) circle (0.2);
\draw[black, fill=white] (10,2) circle (0.2);
\draw[black, fill=white] (11,2) circle (0.2);
\draw[black, fill=white] (10,3) circle (0.2);

\draw (10,3) node{$12$};
\draw (10,2) node{$10$};
\draw (11,2) node{$11$};
\draw (5,2) node{$8$};
\draw (9,1) node{$5$};
\draw (10,1) node{$6$};
\draw (5,1) node{$2$};
\draw (7,0) node{$1$};
\draw (11,1) node{$7$};
\draw (6,1) node{$3$};
\draw (7,1) node{$4$};
\draw (9,2) node{$15$};
\draw (6,2) node{$9$};
\draw (7,2) node{$13$};
\draw (8,2) node{$14$};

\end{tikzpicture}
\caption{\label{bdrex} Bi-decorated rooted hypertrees.}
\end{figure}
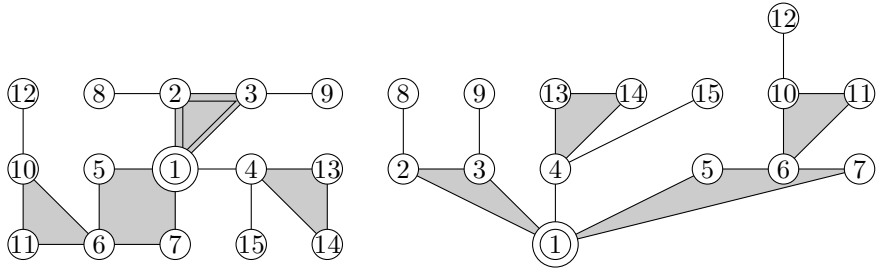

	\subsection{Relations}

		\subsubsection{Dissymetry principle}

We use the dissymetry principle of Proposition \ref{principe de dissymétrie}. As the decoration around vertices is defined similarly for all types of hypertrees, the dissymetry principle is still true for bi-decorated hypertrees.

\begin{prop}[Dissymmetry principle for bi-decorated hypertrees] \label{pdysbi} Given two species $\mathcal{S}_e$ and $\mathcal{S}_v$, the following relation holds: 
\begin{equation}
\hsb+\hsbpa=\hsbp+\hsba.
\end{equation}
\end{prop}

		\subsubsection{Functional equations}
		
The previous species are linked by the following proposition:

\begin{prop} \label{décompbi} Let consider the species $\mathcal{S}_e$, such that $\mathcal{S}_e\left(\emptyset\right) = \emptyset$ and $\lvert \mathcal{S}_e\left(\{1\}\right)\rvert = 0$, and $\mathcal{S}_v$, such that $\mathcal{S}_v\left(\emptyset\right) = \emptyset$ and $\lvert \mathcal{S}_v\left(\{1\}\right) \rvert = 1$. The species $\hsb$, $\hsbp$, $\hsbtp$, $\hsba$, $\hsbpa$ and $\hsbc$ satisfy: 
\begin{align}
t \hsbp & = X + X \times \left( \mathcal{S}_v \circ t \hsbc \right), \\
t \hsbtp & = X + X \times \left(\left(\mathcal{S}'_v -1\right) \circ t \hsbc \right) \label{eq sp},\\
\hsbc & =\mathcal{S}_e' \circ t \hsbtp \label{eq c},\\
\hsba &=\mathcal{S}_e  \circ t \hsbtp, \\
\hsbpa & = \hsbc \times t \hsbtp.
\end{align}
\end{prop}

\begin{proof}
\begin{itemize}
\item The case of only one vertex has already been established in Proposition \ref{décomp}. Otherwise, we separate the label of the root: it gives $X$. It remains a hypertree with a gap contained in different edges, at least one. By definition, those edges has an $\mathcal{S}_v$-structure. Forgetting this structure,we obtain a non-empty set of hollow hypertree with edges decorated by $\mathcal{S}_e$ and vertices decorated by $\mathcal{S}_v$.

\item This case only differs from the preceding case by the structure around the root, which is a $\mathcal{S}'_v$-structure and not a $\mathcal{S}_v$-structure.

\item The third relation is obtained by pointing the vertices in the hollow edge and breaking it: we obtain a non-empty forest of rooted edge-decorated hypertrees. The set of roots is a $\mathcal{S}'_e$-structure and induces this structure on the set of trees: we obtain a $\mathcal{S}'_e$-structure in which all elements are rooted edge-decorated hypertrees. As this operation is reversible and does not depend on the labels of the hollow hypertree, this is an isomorphism of species.

\item The fourth relation is obtained by pointing the vertices in the pointed edge and breaking it: we obtain a non-empty forest of rooted edge-decorated hypertrees. The set of roots is a $\mathcal{S}_e$-structure and induces this structure on the set of trees: we obtain a $\mathcal{S}_e$-structure in which all elements are rooted edge-decorated hypertrees. As this operation is reversible and does not depend on the labels of the hollow hypertree, this is an isomorphism of species.

\item The last equation is obtained by decorating around the vertices of the two parts of the last equality of Proposition \ref{décomp}.\qedhere
\end{itemize}
\end{proof}

\begin{cor} 
Using the equations \eqref{eq sp} and \eqref{eq c} of Proposition \ref{décompbi}, we obtain: 
\begin{align}
\hsbc & =\mathcal{S}_e' \circ \left( X + X \times \left(\mathcal{S}'_v -1\right) \circ t \hsbc \right), \\
t \hsbtp & = X + X \times \left(\mathcal{S}'_v -1\right) \circ t \mathcal{S}_e'  \circ t \hsbtp.
\end{align}
\end{cor}
	
	\subsection{Friendly cases of Bi-decorated hypertrees and link with the hypertree poset}

In the article \cite{mar1}, we link the character of the action of the symmetric group on Whitney homology of the hypertree poset with the symmetric function $\operatorname{HAL}$ defined by F. Chapoton in his article \cite{ChHyp}. We now give a combinatorial interpretation of $\operatorname{HAL}$ in terms of bi-decorated hypertrees.

In this section, in order to keep the notations of \cite{ChHyp}, we will use the exponents a,p and pa in the same way as the letters e,r and re, to denote respectively edge-pointed, rooted and rooted edge-pointed.

As this series has been inspired by the one for \emph{cyclic hypertrees}, we first study the link between bi-decorated hypertrees and this series.

			\subsubsection{Cyclic hypertrees}

In the article \cite{ChHyp}, \emph{cyclic hypertrees} are defined as hypertrees with, for every vertex $v$, a cyclic order on the edges containing $v$. The associated species is denoted by $\operatorname{HAC}$. This definition corresponds to the one of bi-decorated hypertrees obtained by taking the species $\comm-X$ for $\mathcal{S}_e$ and the cycle species $\operatorname{Cycle}$ for $\mathcal{S}_v$.
		
		As for usual hypertrees, we consider rooted cyclic hypertrees, edge-pointed cyclic hypertrees and edge-pointed rooted cyclic hypertrees and respectively write $\operatorname{HAC}^p$, $\operatorname{HAC}^a$ and $\operatorname{HAC}^{pa}$ for the associated species. Considering the definitions, $\operatorname{HAC}^a$ and $\hsya$, and also $\operatorname{HAC}^{pa}$ and $\hsypa$ are isomorphic because in these definitions, pointing and decoration commute. Let us compare the relations to find once more a link between the different types of hypertrees.

On the one hand, the species defined by F. Chapoton satisfy the following relations: 
\begin{equation*}
\operatorname{HAC}^{pa} = t^{-1}{X} \times \left( \assoc \circ t \operatorname{YC} \right),
\end{equation*}
with $\operatorname{YC}$ defined by: 
\begin{align*}
\operatorname{YC} & = \comm \circ\left(X + t \operatorname{HAC}^{pa}\right), \\
\operatorname{HAC}^a & =\left(\comm - X\right) \circ\left(X + t \operatorname{HAC}^{pa}\right), \\
\operatorname{HAC}^p & = t^{-1} X \cycle \circ t \operatorname{YC}, \\
\operatorname{HAC} & = \operatorname{HAC}^a + \operatorname{HAC}^p -\operatorname{HAC}^{pa}.
\end{align*}

Whereas, on the other hand, the bi-decorated hypertrees satisfy: 
\begin{align*}
t \hsyp & = X + X \times \left( \cycle \circ t \hsyc \right), \\
t \hsytp & = X + X \times \left( \assoc \circ t \hsyc \right), \\
\hsyc & = \comm \circ t \hsytp, \\
\hsya &=\left(\comm - X\right) \circ t \hsytp, \\
\hsypa & = \hsyc \times t \hsytp, \\
\hsy & = \hsyp + \hsya - \hsypa.
\end{align*}

Comparing these equations, we obtain the following relations: 
\begin{align*}
\operatorname{HAC}^{a} & =\hsya, \\
X + t \operatorname{HAC}^{pa} & = t \hsytp, \\
\operatorname{YC} & = \hsyc, \\
\operatorname{HAC}^{pa} & = \hsypa, \\
\operatorname{HAC}^p & = t^{-1}{X} + \hsyp.
\end{align*}

Let us now study the case of the $\operatorname{HAL}$ series.

				\subsubsection{The hypertree poset}
				
		The series $\operatorname{HAL}$, $\operatorname{HAL}^p$, $\operatorname{HAL}^a$ and $\operatorname{HAL}^{pa}$ defined in \cite{ChHyp} satisfy the following relations: 
		
\begin{align*}
t \operatorname{HAL}^{pa} & =  p_1 \times \Sigma \assoc \circ t \comm \circ\left(p_1 +\left(-t\right) \operatorname{HAL}^{pa}\right), \\
t \operatorname{HAL}^p & =  p_1 \times \Sigma \lie \circ t \comm \circ\left(p_1 +\left(-t\right) \operatorname{HAL}^{pa}\right), \\
\operatorname{HAL}^a & =\left(\comm - p_1\right) \circ\left(p_1 + (-t) \operatorname{HAL}^{pa}\right), \\
\operatorname{HAL} & = \operatorname{HAL}^a + \operatorname{HAL}^p -\operatorname{HAL}^{pa}.
\end{align*}

We link these series with the cycle index series of bi-decorated hypertrees obtained by taking the species $\comm - X$ for $\mathcal{S}_e$ and the species $\Sigma \lie $ for $\mathcal{S}_v$.

For this choice of decorations, the cycle index series of bi-decorated hypertrees satisfy: 
\begin{align*}
t \hszp & = X + X \times \Sigma \lie \circ t \comm \circ t \hsztp, \\
t \hsztp & = X + X \times \left( -\Sigma \assoc \right) \circ t \comm \circ t \hsztp, \\
\hsza &=\left(\comm - X\right) \circ t \hsztp, \\
\hszpa & = \left( \comm \circ t \hsztp \right) \times t \hsztp, \\
\hsz & = \hszp + \hsza - \hszpa.
\end{align*}

Comparing these equations, we obtain the following relations: 
\begin{align*}
p_1 +\left(-t\right) \operatorname{HAL}^{pa} & = t \hsztp, \\
\operatorname{HAL}^a & = \hsza, \\
\operatorname{HAL}^{pa} & = \hszpa, \\
\operatorname{HAL}^p & = t^{-1} p_1 + \hszp, \\
\operatorname{HAL} & = t^{-1} p_1 + \hsz.
\end{align*}

Therefore, the character of the action of the symmetric group on the Whitney homology of the hypertree poset is the same as the character of the action of the symmetric group on the set of hypertrees whose vertices have their neighbourhood decorated by $\Sigma \lie$.

\section{Reminder on cycle index series}	

\label{Annexe ind cycl}
	
	Let ${F}$ be a species. We can associate a formal power series to it: its cycle index series. The reader can consult the book \cite{BLL} for a reference on this subject. This formal power series is a symmetric function defined as follows:
	
	\begin{defi}
	The \textbf{cycle index series} of a species $F$ is the formal power series in an infinite number of variables $\left(p_1, p_2, p_3, \ldots \right)$ defined by: 
	\begin{equation*}
	\textbf{Z}_F\left(p_1, p_2, p_3, \ldots \right)= \sum_{n \geq 0} \frac{1}{n!} \left( \sum_{\sigma \in \mathfrak{S}_n} \# F^\sigma  \prod_{i \geq 1}p_i^{\sigma_i} \right),
	\end{equation*}
	where $F^\sigma$ stands for the set of $F$-structures fixed under the action of $\sigma$ and where $\sigma_i$ is the number of cycles of length $i$ in the decomposition of $\sigma$ into disjoint cycles.
		\end{defi}

For instance, the cycle index series of the species $X$ of singletons is $p_1$.
		
		We can define the following operations on cycle index series.
		
		\begin{defi}
		The operations $+$ and $\times$ on cycle index series are the same as on formal series.

		For $f=f\left(p_1, p_2, \ldots \right)$ and $g=g\left(p_1, p_2, \ldots \right)$, plethystic substitution $f \circ g$ is defined by: 
		\begin{equation*}
		f \circ g\left(p_1, p_2, \ldots \right) = f\left(g\left(p_1, p_2, p_3, \ldots\right), g\left(p_2, p_4, p_6, \ldots\right), \ldots, g\left(p_{k}, p_{2k}, p_{3k}, \ldots\right), \ldots\right).
		\end{equation*}
		It is left-linear.
		\end{defi}

		These operations satisfy: 

		\begin{prop}
		Let $F$ and $G$ be two species. Their cycle index series satisfy: 
		\begin{equation*}
		\begin{array}{rlrl}
		\textbf{Z}_{F+G}&=\textbf{Z}_F+\textbf{Z}_G, & \textbf{Z}_{F \times G}&=\textbf{Z}_F \times \textbf{Z}_G, \\
		\textbf{Z}_{F\circ G}&=\textbf{Z}_F \circ \textbf{Z}_G, & \textbf{Z}_{F'}&= \frac{\partial \textbf{Z}_{F} }{\partial p_1}.
		\end{array}
		\end{equation*}
		\end{prop}
				
		Moreover, we define the following operation: 
		
		\begin{defi} The \emph{suspension} $\Sigma_t$ of a cycle index series $f\left(p_1, p_2, p_3, \ldots\right)$ is defined by: 
		\begin{equation*}
		\Sigma_t f = - \frac{1}{t} f\left(-tp_1, -t^2 p_2, -t^3 p_3, \ldots\right). 
		\end{equation*}
		By convention, we will write $\Sigma$ for the suspension in $t=1$.
		\end{defi}

The suspension satisfies: 
\begin{prop} \label{prop suspension} Let $f$ and $g$ be two cycle index series. They satisfy: 
\begin{itemize}
\item $\Sigma \left( f \circ g \right)=\Sigma f \circ \Sigma g$,
\item $\Sigma \left( f \times g \right)= - \Sigma f \times \Sigma g$,
\item $-\Sigma f = f \circ\left(-p_1\right)$.
\end{itemize}
\end{prop}
\bibliographystyle{alpha}
\bibliography{bibli}

\begin{thebibliography}{JMM07}

\bibitem[AZ04]{PftB}
Martin Aigner and G{\"u}nter~M. Ziegler.
\newblock {\em Proofs from {T}he {B}ook}.
\newblock Springer-Verlag, Berlin, third edition, 2004.
\newblock Including illustrations by Karl H. Hofmann.

\bibitem[Ber89]{Berge}
Claude Berge.
\newblock {\em Hypergraphs}, volume~45 of {\em North-Holland Mathematical
  Library}.
\newblock North-Holland Publishing Co., Amsterdam, 1989.
\newblock Combinatorics of finite sets, Translated from the French.

\bibitem[BLL98]{BLL}
Fran{\c{c}}ois Bergeron, Gilbert Labelle, and Pierre Leroux.
\newblock {\em Combinatorial species and tree-like structures}, volume~67 of
  {\em Encyclopedia of Mathematics and its Applications}.
\newblock Cambridge University Press, Cambridge, 1998.
\newblock Translated from the 1994 French original by Margaret Readdy, With a
  foreword by Gian-Carlo Rota.

\bibitem[Cha02]{OpDiff}
Fr\'ed\'eric Chapoton.
\newblock {Op\'erades diff\'erentielles gradu\'ees sur les simplexes et les
  permuto\`edres. (Differential graded operads related to simplices and
  permutohedra).}
\newblock {\em Bull. Soc. Math. Fr.}, 130(2):233--251, 2002.

\bibitem[Cha05]{ChAOp}
Fr\'ed\'eric Chapoton.
\newblock On some anticyclic operads.
\newblock {\em Algebr. Geom. Topol.}, 5:53--69 (electronic), 2005.

\bibitem[Cha07]{ChHyp}
Fr\'ed\'eric Chapoton.
\newblock Hyperarbres, arbres enracin\'es et partitions point\'ees.
\newblock {\em Homology, Homotopy Appl.}, 9(1):193--212, 2007.

\bibitem[CL01]{ChLiv}
Fr{\'e}d{\'e}ric Chapoton and Muriel Livernet.
\newblock Pre-{L}ie algebras and the rooted trees operad.
\newblock {\em Internat. Math. Res. Notices}, (8):395--408, 2001.

\bibitem[DK07]{DotsKor}
Vladimir~V. Dotsenko and Anton~S. Khoroshkin.
\newblock Character formulas for the operad of a pair of compatible brackets
  and for the bi-{H}amiltonian operad.
\newblock {\em Funktsional. Anal. i Prilozhen.}, 41(1):1--22, 96, 2007.

\bibitem[Dot12]{Dotp}
Vladimir Dotsenko, 2012.
\newblock Private communication.

\bibitem[EFM]{EFM}
Kurusch Ebrahimi-Fard and Dominique Manchon.
\newblock On an extension of knuth's rotation correspondence to reduced planar
  trees.

\bibitem[GK95]{CochGK}
Ezra Getzler and Mikhail~M. Kapranov.
\newblock Cyclic operads and cyclic homology.
\newblock In {\em Geometry, Topology and Physics}, pages 167--201. Press, 1995.

\bibitem[GK05]{GesKal}
Ira~M. Gessel and Louis~H. Kalikow.
\newblock Hypergraphs and a functional equation of {B}ouwkamp and de {B}ruijn.
\newblock {\em J. Combin. Theory Ser. A}, 110(2):275--289, 2005.

\bibitem[JMM06]{JMcCM2}
Craig Jensen, Jon McCammond, and John Meier.
\newblock The integral cohomology of the group of loops.
\newblock {\em Geom. Topol.}, 10:759--784 (electronic), 2006.

\bibitem[JMM07]{JMcCM}
Craig Jensen, Jon McCammond, and John Meier.
\newblock The {E}uler characteristic of the {W}hitehead automorphism group of a
  free product.
\newblock {\em Trans. Amer. Math. Soc.}, 359(6):2577--2595 (electronic), 2007.

\bibitem[MM04]{McCM}
Jon McCammond and John Meier.
\newblock The hypertree poset and the {$l^2$}-{B}etti numbers of the motion
  group of the trivial link.
\newblock {\em Math. Ann.}, 328(4):633--652, 2004.

\bibitem[Oge]{mar1}
B.~Oger.
\newblock Action of the symmetric groups on the homology of the hypertree
  posets.
\newblock {\em Submitted}.

\bibitem[Sta01]{stan}
Richard~P. Stanley.
\newblock {\em Enumerative Combinatorics}.
\newblock Number vol.~2 in Cambridge Studies in Advanced Mathematics. Cambridge
  University Press, 2001.

\bibitem[Zas02]{Zas}
Thomas Zaslavsky.
\newblock Perpendicular dissections of space.
\newblock {\em Discrete Comput. Geom.}, 27(3):303--351, 2002.

\end{thebibliography}

\end{document}